\documentclass{article}

\usepackage{amssymb, amsmath, amsthm, enumitem, multirow, cancel, physics, color, mathtools, xcolor, mathrsfs , hyperref, tkz-euclide, tikz, bigints}

\usepackage{arxiv}

\usepackage[utf8]{inputenc} 
\usepackage[T1]{fontenc}    
\usepackage{hyperref}       
\usepackage{url}            
\usepackage{booktabs}       
\usepackage{amsfonts}       
\usepackage{nicefrac}       
\usepackage{microtype}      
\usepackage{lipsum}

\usepackage{mathtools}

\usepackage{color}
\usepackage{tcolorbox}
\usepackage{amsmath,amsfonts,amssymb,graphicx}
\usepackage{enumitem}
\usepackage{float}
\usepackage{color}
\usepackage{hyperref}
\usepackage{graphicx}
\usepackage{longtable}
\usepackage{algcompatible}
\usepackage{algorithm}\usepackage{caption}
\usepackage{adjustbox}
\usepackage{array}
\usepackage{rotating} 
\usepackage{caption}
\usepackage[font=scriptsize, skip=0cm, labelfont=bf]{caption}
\usepackage{booktabs}
\usepackage[justification=centering]{caption}
\usepackage[title]{appendix}
\usepackage{blindtext}
\usepackage{slashbox}
\usepackage[utf8]{inputenc}
\usepackage[english]{babel}

\newtheorem{theorem}{Theorem}
\newtheorem{lemma}{Lemma}
\newtheorem{corollary}{Corollary}
\newtheorem{assumption}{Assumption}
\usepackage{tikz}

\title{ Neural-trust-region algorithm for unconstrained optimization (Part 1) }

\author{
 Mostafa Rezapour \\
 Department of Mathematics and Statistics\\
 Washington State University\\
  Pullman, WA 99163 \\
  \texttt{mostafa.rezapour@wsu.edu} \\
   \And
 Thomas Asaki \\
  Department of Mathematics and Statistics\\
 Washington State University\\
  Pullman, WA 99163 \\
  \texttt{tasaki@wsu.edu} \\
  }
\begin{document}
\maketitle

\begin{abstract}

In this paper (part 1), we describe a derivative-free trust-region method for solving unconstrained optimization problems. We will discuss a method when we relax the model order assumption and use artificial neural network techniques to build a computationally relatively inexpensive model. We directly find an estimate of the objective function minimizer without explicitly constructing a model function. Therefore, we need to have the neural-network model derivatives, which can be obtained simply through a back-propagation process.

\end{abstract}


\section{Introduction}

There are different iterative numerical procedures, such as line-search and trust-region methods, for finding a local solution of the problem 
\begin{equation} \label{e1}
\min\limits_{x \in \mathbb{R}^{n} } f(x),
\end{equation}
where $f : \mathbb{R}^{n} \rightarrow \mathbb{R}$ is at least twice continuously differentiable and is bounded below \cite{conn2000trust}. The trust-region method was initially introduced by Powell in the 1970s for solving nonlinear optimization problems. At each iterate $x^k$, a basic trust-region first defines a model $m_k(x)$ of the objective function within an appropriate neighborhood of $x^k$ and then finds an approximate solution of the trust-region subproblem
\begin{equation}\label{e2}
s^k=\arg \min\limits_{\| s \| \leq   \Delta_k} { m_k,}
\end{equation}
where $\Delta_k$ is the trust-region radius and $\|\cdot\|$ is any vector norm. In trust-region methods, a quadratic model $m_k$ of the objective function is widely considered
\begin{equation}\label{e3}
m_k (x^k+ s ) = f_k +g^T_k s +\dfrac{1}{2} s^T B_ks, 
\end{equation}
where $f_k=f(x^k)$, $g_k=\nabla f(x^k)$ and $B_k$ is either the Hessian of the objective function at the current point $x^k$ or an approximation of it \cite{conn2000trust}.
To find how good the approximate solution $s^k$ of the trust-region subproblem (\ref{e2}) and the agreement between the model and the objective function in the trust-region are, the following agreement ratio can be used 
\begin{equation}\label{e4}
\rho_k = \frac{f(x^k) - f( x^k + s^k )}{ m_k ( x^k )-m_k( x^k+s^k )}.
\end{equation}
For given constants $0<\eta_1<\eta_2\leq1$, a basic trust-region algorithm works as follows. If $\rho_k \geq \eta_2$, which means there exists a very good agreement between the model and the objective function in the current trust-region, then $x^{k+1}=x^k+s^k$ is accepted as the new trial point and the trust-region radius is enlarged. If $\eta_1\leq \rho_k <\eta_2$, which means the agreement is good, then $x^{k+1}=x^k+s^k$ is accepted as the new point but the trust-region radius remains the same.  If $\rho_k<\eta_1$, which means the agreement is weak, then $x^{k+1}=x^k$ and the trust-region radius is reduced \cite{conn2000trust}.

One of the most widely used methods for solving (\ref{e2}) is the Steihaug-Toint method, which uses the conjugate-gradient algorithm and approximates a solution with minimal computational cost \cite{conn2000trust}. However, Gould et al. \cite {gould1999solving} noted that the obtained approximate solution $s^k$ for (\ref{e2}) by the Steihaug-Toint method may be the first boundary point, which is the point on the boundary of the trust region in the direction of $s^k$. They proposed GLTR method in which Lanczos process and More–Sorensen algorithm are used to reduce $B_k$ to tridiagonal form and solve the obtained reduced subproblem. Erway et al. \cite{erway2009iterative} proposed an extension of Steihaug-Toint method for large-scale optimization problems with two phases, phased-SSM method, in which if in phase 1, the subproblem approximate solution obtained from the modified Steihaug-Toint algorithm lies on the boundary, then phase 2 becomes active and a conjugate-gradient based SSM method is used to solve the constrained problem over a sequence of evolving low-dimensional subspaces.

For general large-scale optimization, where $B_k$ is not a quasi-Newton Hessian, solving subproblem (\ref{e2}) is often computationally expensive. For example, the Moré-Sorensen method solves $(B_k+\sigma I)s = -g_k$ at each iteration by computing the Cholesky factorization of $B_k+\sigma I$ to find a solution $(s^*,\sigma^*)$ that satisfies the optimality conditions for the trust-region subproblem. Quasi-Newton methods, which generate a sequence of matrices which approximate the Hessian (or its inverse) of the objective function, need to store $n(n + 1)/2$ elements for each approximate symmetric matrix in the sequence. Hence, Quasi-Newton methods are not computationally efficient for large-scale optimization problems. Given the number of L-BFGS updates $M$, limited-memory BFGS quasi-Newton methods, which generate matrices using information from the last $M$ iterations, are often used when the second derivative is prohibitively expensive (\cite{byrd1995limited}; \cite{liu1989limited}; \cite{morales2002numerical}; \cite{nocedal1980updating}). However, L-BFGS methods require to solve a system of the form $(B_k+D)q=z$, which is often expensive. The two-loop recursion (\cite{nocedal1980updating};\cite{nocedal2006regularization}) and the recursion formula proposed by Erway et al \cite{erway2012limited} can be used to solve the system.

There are many problems in which derivatives are not available or computationally very expensive. Model-based derivative-free methods are often utilized to solve (\ref{e1}) by replacing a computationally expensive function with one of a computationally cheaper surrogate model. The model is often constructed by applying a multivariate interpolation at the available objective function values $f(x_i)$.  A quadratic objective function model of the form (\ref{e2}) is often chosen, where $g_k$ and $B_k$ are determined by the interpolation process at past points. Given an interpolation set $Y=\{y^1, y^2, \dots, y^p\}$, it is often required that 
\begin{equation}\label{MMarch21}
f(y^i)=m_k(y^i),
\end{equation}
for $i=1, 2, \dots, p$. If we assume the model to be quadratic, then the cardinality of the interpolation set $|Y|=p$ must satisfy 
\begin{equation}
p\geq\frac{1}{2}(n+1)(n+2)
\end{equation}
to ensure that that the quadratic model is entirely determined \cite{conn2000trust}. To ensure the existence and uniqueness of the quadratic interpolant, a basis $\{\phi_i(.)\}$ of the linear space of $n$ dimensional quadratics is chosen and the model is expressed as 
\begin{equation}\label{m1}
m_k(x)=\sum_{k=1}^{p}\alpha_k \phi_k(x).
\end{equation}
So (\ref{MMarch21}) might be expressed as 
\begin{equation}
f(y^i)=\sum_{k=1}^{p}\alpha_k \phi_k(y^i),
\end{equation}
for $i=1, 2, \dots, p$. An interpolation set $Y$ is said to be poised if the interpolation determinant $D(Y)$ is nonzero \cite{conn2000trust},
\begin{equation}\label{ap123}
   D(Y)=det
  \left[ {\begin{array}{cccc}
   \phi_1(y^1) & \phi_2(y^1)& \dots & \phi_p(y^1) \\
   \phi_1(y^2) & \phi_2(y^2)& \dots & \phi_p(y^2) \\
       & & \vdots &  \\
    \phi_1(y^p) & \phi_2(y^p)& \dots & \phi_p(y^p) \\  
  \end{array} } \right] \ne 0.
\end{equation}
So a unique quadratic model can be determined if and only if the interpolation set $Y$ is poised. For an interpolation set $Y$, Algorithm 9.4.1 in \cite{conn2000trust} (CNP procedure in \cite{conn1997convergence}) gives a procedure for constructing the basis of fundamental Newton polynomials. 

Algorithm 9.1 in \cite{nocedal2006numerical} describes a model-based derivative-free trust-region method in which the step acceptance, trust-region update and the interpolation set update strategies are based on the agreement ratio. Given model $m_k$ at iteration $k$, if the agreement between the model and the objective function is good, $\rho_k\geq \eta_1$, then 
\begin{equation}
\Delta_{k+1}=\gamma_2 \Delta_k,
\end{equation}    
and the interpolation set $Y=\{y^1, y^2, \dots,y^-, \dots, y^p\}$ is replaced by 
\begin{equation}
Y^+=\{y^1, y^2, \dots,y^+, \dots, y^p\}, 
\end{equation}
where $\gamma_2 \geq 1$, 
\begin{equation}\label{ap111}
|D(Y^+)| \leq |L(y^+,y^-)|\hspace{.1cm}|D(Y)| 
\end{equation}
and $L(.,.)$ is the Lagrangian polynomial with degree at most two that satisfies
\begin{equation}\label{ap112}
L(y^+,y^-)=
\begin{cases}
1 \hspace{1cm} \text{if } y^+=y^-\\
0 \hspace{1cm} \text{if } y^+\neq y^-.
\end{cases}
\end{equation}
Using (\ref{ap111}) and (\ref{ap112}), we might select $y^-$ to be
\begin{equation}
y^-=\text{arg} \max\limits_{y^i \in Y} L(x^+, y^i),
\end{equation}
where
\begin{equation}
{({x^k})}^+=y^+=x^k+s^k.
\end{equation}
Conn et al. \cite{conn1997convergence} define an adequate geometry of the interpolation set for a quadratic model. For instance to construct a quadratic model using the CNP procedure in \cite{conn1997convergence}, the interpolation set $Y$ is first organized into $d+1=3$ blocks $Y=\{ \mathbb{Y}^{[0]}, \mathbb{Y}^{[1]}, \mathbb{Y}^{[2]} \}$, where the cardinality of $\mathbb{Y}^{[l]}$ is 
\[
|\mathbb{Y}^{[l]}|= {{l+n-1} \choose l}
\]
for $l=0, 1, 2$. Then for each $y^{i,[l]} \in \mathbb{Y}^{[l]}$ a corresponding Newton fundamental polynomial of degree $l$ satisfying
\[ N_ i^{[l]}(y^{j,[m]})=\delta_{ij}\delta_{lm},
\]
for all $y^{j,[m]} \in \mathbb{Y}^{[m]}$ with $m\leq l$. Then the model is constructed as
\[m(x)=\sum_{y^{i,[l]} \in Y} d^{i,[l]}(Y,f)N_ i^{[l]}(x),
\]
where $d^{i,[l]}(Y,f)$ are generalized finite differences applied on $f$. At iterate $x^k$, the given interpolation set $Y=\left\{y^i: \|x^k-y^i \|\leq \Delta_k \text{ for all }i=1, 2, \dots, p \right\}$ is adequate in the current trust-region if the cardinality of $Y$ is at least $n+1$, and for all interpolation points
\[\left |N_ i^{[l]}(y^{j,[l+1])}) \right|\leq \kappa_n, 
\]
and 
\[\left |N_ i^{[z]}(x)\right|\leq \kappa_n,
\]
for all $ i=1, 2, \dots, \mathbb{Y}^{[l]}, j=1, 2, \dots, \mathbb{Y}^{[l+1]}, l=0, 1, \dots, d$ (for a quadratic model $d=2$), $\kappa_n >2^{\mathbb{Y}^{[d]}}$, and $x$ is any point in the current trust-region \cite{conn1997convergence}. 

\noindent \textbf{Example 1 (A complete Newton polynomial basis)} The problem is to do a quadratic interpolation 
on the interpolation set $Y=\left \{Y^{[0]}, Y^{[1]}, Y^{[2]} \right \}$, where
\begin{equation}
\begin{aligned}
Y^{[0]}=\left \{ (0,0) \right \},\hspace{1cm}\\
Y^{[1]}=\left \{  (\frac{1}{2},0), (0,\frac{1}{2})    \right \},\\
Y^{[2]}=\left \{   (1,0), (\frac{1}{2},\frac{1}{2}), (0,1)   \right \}.
\end{aligned}
\end{equation}
We now apply Algorithm 9.4.1 in \cite{conn2000trust} to the initial quadratic polynomial basis
\begin{equation}
\beta=\left \{ 1, x_1, x_2, x_1^2, x_1x_2, x_2^2\right \}.
\end{equation}
We set
\begin{equation}
\begin{aligned}
N_1^{[0]}(x_1,x_2)=1,\hspace{4cm}\\
N_1^{[1]}(x_1,x_2)=x_1, \hspace{1cm}N_2^{[1]}(x_1,x_2)=x_2,\hspace{2cm}\\
N_1^{[2]}(x_1,x_2)=x_1^2,\hspace{1cm} N_2^{[2]}(x_1,x_2)=x_1x_2, \hspace{1cm} N_3^{[2]}(x_1,x_2)=x_2^2,\\
Y_{\text{temp}}=\left\{ \right\}.\hspace{4.7cm}
\end{aligned}
\end{equation}

\noindent \textbf{Calculating the Newton fundamental polynomial $i=0$ and $j=1$:} 

Since 
\begin{equation}
\begin{aligned}
N_1^{[0]}(0,0)=1,\hspace{3cm}\\
N_1^{[1]}(\frac{1}{2},0)=\frac{1}{2}, \hspace{1cm}N_2^{[1]}(0,\frac{1}{2})=\frac{1}{2},\hspace{1cm}\\
N_1^{[2]}(1,0)=1,\hspace{1cm} N_2^{[2]}(\frac{1}{2},\frac{1}{2})=\frac{1}{4}, \hspace{1cm} N_3^{[2]}(0,1)=1,\\
\end{aligned}
\end{equation}
we can choose $y_j^{[i]} \in \left \{ (0,0), (\frac{1}{2},0), (0,\frac{1}{2}), (1,0), (\frac{1}{2},\frac{1}{2}), (0,1)\right \}$. We choose $y_1^{[0]}=(0,0)$, and set 
\begin{equation*}
Y_{\text{temp}}=\left \{(0,0) \right \}.
\end{equation*}
Then we normalize $N_1^{[0]}$,
\begin{equation*}
N_1^{[0]}(x_1,x_2)=1.
\end{equation*}
Newton polynomials will be updated by
\begin{equation}
\begin{aligned}
{}^*N_1^{[0]}(x_1,x_2)=1,\hspace{4cm}\\
N_1^{[1]}(x_1,x_2)=x_1, \hspace{.4cm}N_2^{[1]}(x_1,x_2)={x_2},\hspace{2.5cm}\\
N_1^{[2]}(x_1,x_2)=x_1^2,\hspace{.4cm} N_2^{[2]}(x_1,x_2)=x_2^2, \hspace{.4cm} N_3^{[2]}(x_1,x_2)=x_1x_2.\\
\end{aligned}
\end{equation}

\noindent \textbf{Calculating the Newton fundamental polynomial $i=1$ and $j=1$:} 
\begin{equation}
\begin{aligned}
N_1^{[1]}(\frac{1}{2},0)=\frac{1}{2}, \hspace{1cm}N_2^{[1]}(0,\frac{1}{2})=\frac{1}{2},\hspace{2cm}\\
N_1^{[2]}(1,0)=1,\hspace{1cm} N_2^{[2]}(\frac{1}{2},\frac{1}{2})=\frac{1}{4}, \hspace{1cm} N_3^{[2]}(0,1)=1,\\
\end{aligned}
\end{equation}
we can choose a $y_j^{[i]} \in \left \{  (\frac{1}{2},0), (0,\frac{1}{2}), (1,0), (\frac{1}{2},\frac{1}{2}), (0,1)\right \}$. We choose $y_1^{[1]}= (\frac{1}{2},0)$, and set 
\begin{equation*}
Y_{\text{temp}}=\left \{(0,0),  (\frac{1}{2},0) \right \}.
\end{equation*}
Then we normalize $N_1^{[1]}$,
\begin{equation*}
N_1^{[1]}(x_1,x_2)=2x_1.
\end{equation*}
Newton polynomials will be updated by
\begin{equation}
\begin{aligned}
{}^*N_1^{[0]}(x_1,x_2)=1,\hspace{4cm}\\
{}^*N_1^{[1]}(x_1,x_2)=2x_1, \hspace{.4cm}N_2^{[1]}(x_1,x_2)={x_2},\hspace{2.5cm}\\
N_1^{[2]}(x_1,x_2)=x_1^2- \frac{1}{2}x_1,\hspace{.4cm} N_2^{[2]}(x_1,x_2)=x_1x_2, \hspace{.4cm} N_3^{[2]}(x_1,x_2)=x_2^2.\\
\end{aligned}
\end{equation}

\noindent \textbf{Calculating the Newton fundamental polynomial $i=1$ and $j=2$:} 
\begin{equation}
\begin{aligned}
N_2^{[1]}(0,\frac{1}{2})=\frac{1}{2},\hspace{4cm}\\
N_1^{[2]}(1,0)=\frac{1}{2},\hspace{1cm} N_2^{[2]}(\frac{1}{2},\frac{1}{2})=\frac{1}{4}, \hspace{1cm} N_3^{[2]}(0,1)=1,\\
\end{aligned}
\end{equation}
we can choose a $y_j^{[i]} \in \left \{  (0,\frac{1}{2}), (1,0), (\frac{1}{2},\frac{1}{2}), (0,1)\right \}$. We choose $y_2^{[1]}=(0,\frac{1}{2})$, and set 
\begin{equation*}
Y_{\text{temp}}=\left \{(0,0),  (\frac{1}{2},0),  (0,\frac{1}{2})  \right \}.
\end{equation*}
Then we normalize $N_2^{[1]}$,
\begin{equation*}
N_1^{[1]}(x_1,x_2)=2{x_2}.
\end{equation*}
Newton polynomials are updated by
\begin{equation}
\begin{aligned}
{}^*N_1^{[0]}(x_1,x_2)=1,\hspace{4cm}\\
{}^*N_1^{[1]}(x_1,x_2)=2x_1, \hspace{.4cm}{}^*N_2^{[1]}(x_1,x_2)=2{x_2},\hspace{2.5cm}\\
N_1^{[2]}(x_1,x_2)=x_1^2- \frac{1}{2}x_1,\hspace{.4cm} N_2^{[2]}(x_1,x_2)=x_1x_2, \hspace{.4cm} N_3^{[2]}(x_1,x_2)=x_2^2- \frac{1}{2}x_2.\\
\end{aligned}
\end{equation}

\noindent \textbf{Calculating the Newton fundamental polynomial $i=2$ and $j=1$:} 
\begin{equation}
\begin{aligned}
N_1^{[2]}(1,0)=\frac{1}{2},\hspace{1cm} N_2^{[2]}(\frac{1}{2},\frac{1}{2})=\frac{1}{4}, \hspace{1cm} N_3^{[2]}(0,1)=\frac{1}{2},\\
\end{aligned}
\end{equation}
we can choose a $y_j^{[i]} \in \left \{(1,0), (\frac{1}{2},\frac{1}{2}), (0,1)\right \}$. We choose $y_1^{[2]}=(1,0)$, and set 
\begin{equation*}
Y_{\text{temp}}=\left \{(0,0),  (\frac{1}{2},0),  (0,\frac{1}{2}), (1,0)  \right \}.
\end{equation*}
Then we normalize $N_1^{[2]}$,
\begin{equation*}
N_1^{[2]}(x_1,x_2)=2x_1^2-x_1.
\end{equation*}
Newton polynomials will be updated by
\begin{equation}
\begin{aligned}
{}^*N_1^{[0]}(x_1,x_2)=1,\hspace{4cm}\\
{}^*N_1^{[1]}(x_1,x_2)=2x_1, \hspace{.4cm}{}^*N_2^{[1]}(x_1,x_2)=2{x_2},\hspace{2.5cm}\\
{}^*N_1^{[2]}(x_1,x_2)=2x_1^2-x_1,\hspace{.4cm} N_2^{[2]}(x_1,x_2)=x_1x_2, \hspace{.4cm} N_3^{[2]}(x_1,x_2)=x_2^2- \frac{1}{2}x_2.\\
\end{aligned}
\end{equation}

\noindent \textbf{Calculating the Newton fundamental polynomial $i=2$ and $j=2$:} 
\begin{equation}
\begin{aligned}
\hspace{1cm} N_2^{[2]}(\frac{1}{2},\frac{1}{2})=\frac{1}{4}, \hspace{1cm} N_3^{[2]}(0,1)=\frac{1}{2},\\
\end{aligned}
\end{equation}
we can choose a $y_j^{[i]} \in \left \{(\frac{1}{2},\frac{1}{2}), (0,1)\right \}$. We choose $y_2^{[2]}=(\frac{1}{2},\frac{1}{2})$, and set 
\begin{equation*}
Y_{\text{temp}}=\left \{(0,0),  (\frac{1}{2},0),  (0,\frac{1}{2}), (1,0), (\frac{1}{2},\frac{1}{2})  \right \}.
\end{equation*}
Then we normalize $N_2^{[2]}$,
\begin{equation*}
N_2^{[2]}(x_1,x_2)=4x_1x_2.
\end{equation*}
Newton polynomials are updated by
\begin{equation}
\begin{aligned}
{}^*N_1^{[0]}(x_1,x_2)=1,\hspace{4cm}\\
{}^*N_1^{[1]}(x_1,x_2)=2x_1, \hspace{.4cm}{}^*N_2^{[1]}(x_1,x_2)=2{x_2},\hspace{2.5cm}\\
{}^*N_1^{[2]}(x_1,x_2)=2x_1^2-x_1,\hspace{.4cm} {}^*N_2^{[2]}(x_1,x_2)=4x_1x_2, \hspace{.4cm} N_3^{[2]}(x_1,x_2)=x_2^2- \frac{1}{2}x_2.\\
\end{aligned}
\end{equation}

\noindent \textbf{Calculating the Newton fundamental polynomial $i=3$ and $j=2$:} 
\begin{equation}
\begin{aligned}
\hspace{1cm} N_3^{[2]}(0,1)=\frac{1}{2},\\
\end{aligned}
\end{equation}
we can choose a $y_j^{[i]} \in \left \{ (0,1)\right \}$. We choose $y_3^{[2]}=(0,1)$, and set 
\begin{equation*}
Y_{\text{temp}}=\left \{(0,0),  (\frac{1}{2},0),  (0,\frac{1}{2}), (1,0), (\frac{1}{2},\frac{1}{2}), (0,1) \right \}.
\end{equation*}
Then we normalize $N_3^{[2]}$,
\begin{equation*}
N_2^{[2]}(x_1,x_2)=2x_2^2-x_2.
\end{equation*}
Finally all Newton polynomials will be updated by
\begin{equation}
\begin{aligned}
{}^*N_1^{[0]}(x_1,x_2)=1,\hspace{4cm}\\
{}^*N_1^{[1]}(x_1,x_2)=2x_1, \hspace{.4cm}{}^*N_2^{[1]}(x_1,x_2)=2{x_2},\hspace{2.5cm}\\
{}^*N_1^{[2]}(x_1,x_2)=2x_1^2-x_1,\hspace{.4cm} {}^*N_2^{[2]}(x_1,x_2)=4x_1x_2, \hspace{.4cm} {}^*N_3^{[2]}(x_1,x_2)=2x_2^2-x_2.\\
\end{aligned}
\end{equation}

\noindent \textbf{Example 2 (An incomplete Newton polynomial basis when $Y$ is not poised)} The problem is to do a quadratic interpolation 
on the interpolation set $Y=\left \{Y^{[0]}, Y^{[1]}, Y^{[2]} \right \}$, where
\begin{equation}
\begin{aligned}
Y^{[0]}=\left \{ (0,0) \right \},\hspace{1cm}\\
Y^{[1]}=\left \{  (1,0), (0,2)    \right \},\\
Y^{[2]}=\left \{   (3,0), (1,2), (2,1)   \right \}.
\end{aligned}
\end{equation}
We now apply Algorithm 9.4.1 in \cite{conn2000trust} to the initial quadratic polynomial basis
\begin{equation}
\beta=\left \{ 1, x_1, x_2, x_1^2, x_2^2, x_1x_2\right \}.
\end{equation}
We set
\begin{equation}
\begin{aligned}
N_1^{[0]}(x_1,x_2)=1,\hspace{4cm}\\
N_1^{[1]}(x_1,x_2)=x_1, \hspace{1cm}N_2^{[1]}(x_1,x_2)=x_2,\hspace{2cm}\\
N_1^{[2]}(x_1,x_2)=x_1^2,\hspace{1cm} N_2^{[2]}(x_1,x_2)=x_2^2, \hspace{1cm} N_3^{[2]}(x_1,x_2)=x_1x_2,\\
Y_{\text{temp}}=\left\{ \right\}.\hspace{3.6cm}
\end{aligned}
\end{equation}

\noindent \textbf{Calculating the Newton fundamental polynomial $i=0$ and $j=1$:} 

Since 
\begin{equation}
\begin{aligned}
N_1^{[0]}(0,0)=1,\hspace{3cm}\\
N_1^{[1]}(1,0)=1, \hspace{1cm}N_2^{[1]}(0,2)=2,\hspace{1cm}\\
N_1^{[2]}(3,0)=9,\hspace{1cm} N_2^{[2]}(1,2)=4, \hspace{1cm} N_3^{[2]}(2,1)=2,\\
\end{aligned}
\end{equation}
we can choose $y_j^{[i]} \in \left \{ (0,0), (1,0), (0,2), (3,0), (1,2), (2,1)\right \}$. We choose $y_1^{[0]}=(0,0)$, and set 
\begin{equation*}
Y_{\text{temp}}=\left \{(0,0) \right \}.
\end{equation*}
Then we normalize $N_1^{[0]}$,
\begin{equation*}
N_1^{[0]}(x_1,x_2)=1.
\end{equation*}
Newton polynomials are updated as
\begin{equation}
\begin{aligned}
{}^*N_1^{[0]}(x_1,x_2)=1,\hspace{4cm}\\
N_1^{[1]}(x_1,x_2)=x_1, \hspace{.4cm}N_2^{[1]}(x_1,x_2)={x_2},\hspace{2.5cm}\\
N_1^{[2]}(x_1,x_2)=x_1^2,\hspace{.4cm} N_2^{[2]}(x_1,x_2)=x_2^2, \hspace{.4cm} N_3^{[2]}(x_1,x_2)=x_1x_2.\\
\end{aligned}
\end{equation}

\noindent \textbf{Calculating the Newton fundamental polynomial $i=1$ and $j=1$:} 
\begin{equation}
\begin{aligned}
N_1^{[1]}(1,0)=1, \hspace{1cm}N_2^{[1]}(0,2)=2,\hspace{1cm}\\
N_1^{[2]}(3,0)=9,\hspace{1cm} N_2^{[2]}(1,2)=4, \hspace{1cm} N_3^{[2]}(2,1)=2,\\
\end{aligned}
\end{equation}
we can choose a $y_j^{[i]} \in \left \{  (1,0), (0,2), (3,0), (1,2), (2,1)\right \}$. We choose $y_1^{[1]}=(1,0)$, and set 
\begin{equation*}
Y_{\text{temp}}=\left \{(0,0), (1,0) \right \}.
\end{equation*}
Then we normalize $N_1^{[1]}$,
\begin{equation*}
N_1^{[1]}(x_1,x_2)=x_1.
\end{equation*}
Finally all Newton polynomials will be updated by
\begin{equation}
\begin{aligned}
{}^*N_1^{[0]}(x_1,x_2)=1,\hspace{4cm}\\
{}^*N_1^{[1]}(x_1,x_2)=x_1, \hspace{.4cm}N_2^{[1]}(x_1,x_2)={x_2},\hspace{2.5cm}\\
N_1^{[2]}(x_1,x_2)=x_1^2-x_1,\hspace{.4cm} N_2^{[2]}(x_1,x_2)=x_2^2, \hspace{.4cm} N_3^{[2]}(x_1,x_2)=x_1x_2.\\
\end{aligned}
\end{equation}

\noindent \textbf{Calculating the Newton fundamental polynomial $i=1$ and $j=2$:} 
\begin{equation}
\begin{aligned}
N_2^{[1]}(0,2)=2,\hspace{4cm}\\
N_1^{[2]}(3,0)=6,\hspace{1cm} N_2^{[2]}(1,2)=4, \hspace{1cm} N_3^{[2]}(2,1)=2,\\
\end{aligned}
\end{equation}
we can choose a $y_j^{[i]} \in \left \{  (0,2), (3,0), (1,2), (2,1)\right \}$. We choose $y_2^{[1]}=(0,2)$, and set 
\begin{equation*}
Y_{\text{temp}}=\left \{(0,0), (1,0), (0,2) \right \}.
\end{equation*}
Then we normalize $N_2^{[1]}$,
\begin{equation*}
N_1^{[1]}(x_1,x_2)=\frac{x_2}{2}.
\end{equation*}
Newton polynomials are updated as
\begin{equation}
\begin{aligned}
{}^*N_1^{[0]}(x_1,x_2)=1,\hspace{4cm}\\
{}^*N_1^{[1]}(x_1,x_2)=x_1, \hspace{.4cm}{}^*N_2^{[1]}(x_1,x_2)=\frac{x_2}{2},\hspace{2.5cm}\\
N_1^{[2]}(x_1,x_2)=x_1^2-x_1,\hspace{.4cm} N_2^{[2]}(x_1,x_2)=x_2^2-2x_2, \hspace{.4cm} N_3^{[2]}(x_1,x_2)=x_1x_2.\\
\end{aligned}
\end{equation}

\noindent \textbf{Calculating the Newton fundamental polynomial $i=2$ and $j=1$:} 
\begin{equation}
\begin{aligned}
N_1^{[2]}(3,0)=6,\hspace{1cm} N_2^{[2]}(1,2)=0, \hspace{1cm} N_3^{[2]}(2,1)=2,\\
\end{aligned}
\end{equation}
we can choose a $y_j^{[i]} \in \left \{  (3,0), (2,1)\right \}$. We choose $y_1^{[2]}=(3,0)$, and set 
\begin{equation*}
Y_{\text{temp}}=\left \{(0,0), (1,0), (0,2), (3,0) \right \}.
\end{equation*}
Then we normalize $N_1^{[2]}$,
\begin{equation*}
N_1^{[2]}(x_1,x_2)=\frac{x_1^2-x_1}{6}.
\end{equation*}
Newton polynomials are updated by
\begin{equation}
\begin{aligned}
{}^*N_1^{[0]}(x_1,x_2)=1,\hspace{4cm}\\
{}^*N_1^{[1]}(x_1,x_2)=x_1, \hspace{.4cm}{}^*N_2^{[1]}(x_1,x_2)=\frac{x_2}{2},\hspace{2.5cm}\\
{}^*N_1^{[2]}(x_1,x_2)=\frac{1}{6}(x_1^2-x_1),\hspace{.4cm} N_2^{[2]}(x_1,x_2)=x_2^2-2x_2, \hspace{.4cm} N_3^{[2]}(x_1,x_2)=x_1x_2.\\
\end{aligned}
\end{equation}

\noindent \textbf{Calculating the Newton fundamental polynomial $i=2$ and $j=3$:} 
\begin{equation}
\begin{aligned}
N_2^{[2]}(1,2)=0,\hspace{.4cm} N_3^{[2]}(2,1)=2,\\
\end{aligned}
\end{equation}
Then we normalize $N_3^{[2]}$,
\begin{equation*}
N_3^{[2]}(x_1,x_2)=\frac{1}{2}({x_1x_1}).
\end{equation*}
Finally all Newton polynomials will be updated by
\begin{equation}
\begin{aligned}
{}^*N_1^{[0]}(x_1,x_2)=1,\hspace{6cm}\\
{}^*N_1^{[1]}(x_1,x_2)=x_1, \hspace{.4cm}{}^*N_2^{[1]}(x_1,x_2)=\frac{x_2}{2},\hspace{4.5cm}\\
{}^*N_1^{[2]}(x_1,x_2)=\frac{1}{6}(x_1^2-x_1),\hspace{.4cm} N_2^{[2]}(x_1,x_2)=x_2^2-2x_2 \hspace{.4cm} {}^*N_3^{[2]}(x_1,x_2)=\frac{1}{2}({x_1x_1}).\\
\end{aligned}
\end{equation}

\noindent \textbf{Calculating the Newton fundamental polynomial $i=2$ and $j=2$:} 
\begin{equation}
\begin{aligned}
N_2^{[2]}(1,2)=0,\\
\end{aligned}
\end{equation}
which means the basis of Newton polynomials is incomplete, because 
\begin{equation}
N_2^{[2]}(y_2^{[2]})=0.\\
\end{equation}

Theorem 9.4.1 in \cite{conn2000trust} indicates that the model $m_k$ that is constructed based on the Newton fundamental polynomials, 
\begin{equation}\label{may22222}
m_k(x)=\sum_{i=0}^{d} \sum_{j=1}^{|Y^{[i]}|} \lambda_i(y_j^{[i]})N_j^{[i]}(x)
\end{equation}
for $ i= 0, 1, \dots, d$ and $j=1, \dots, |Y^{[i]}|$ is well defined and satisfies the interpolation initial conditions, where
\begin{equation}
\begin{aligned}
\lambda_0(x)=f(x) \hspace{2cm},\\
\lambda_{i+1}(x)=\lambda_{i}(x)-\sum_{j=1}^{|Y^{[i]}|} \lambda_i(y_j^{[i]})N_j^{[i]}(x),
\end{aligned}
\end{equation}
for  $ i= 0, 1, \dots, d-1$.

\noindent \textbf{Example 3 (Constructing a quadratic model using the Newton fundamental polynomials)} The problem is to construct a model $m(x)$ of the objective function 
\begin{equation}
f(x_1,x_2)=\left(x_1-2\right)^4+\left(x_2-1\right)^3+e^{x_1+x_2}
\end{equation}
on the interpolation set $Y=\left \{ (0,0), (1,0), (0,1), (2,0), (1,1), (0,2) \right \}$.
\begin{figure}[h!]
\centering
\includegraphics[width=\textwidth]{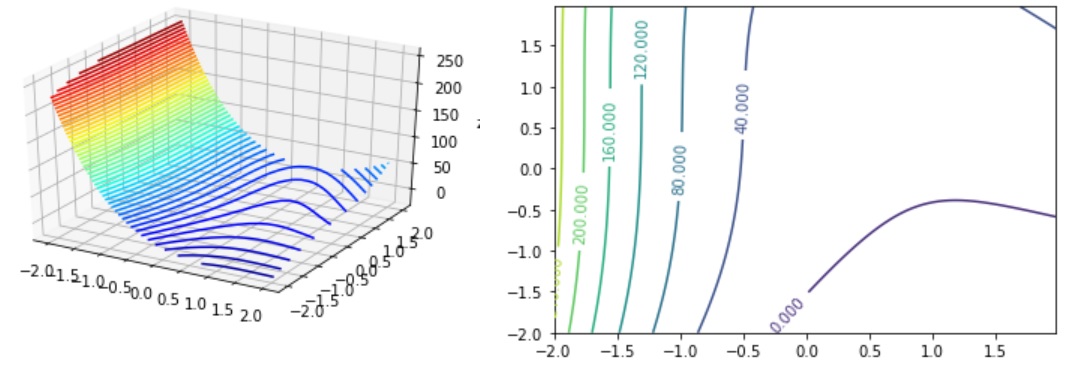}
\caption{\label{} Illustration of $f(x_1,x_2)=\left(x_1-2\right)^4+\left(x_2-1\right)^3+e^{x_1+x_2}$ and its level curves.}
\end{figure}
We organize the points of $Y$ into $Y=\left \{Y^{[0]}, Y^{[1]}, Y^{[2]} \right \}$, where
\begin{equation}
\begin{aligned}
Y^{[0]}=\left \{ (0,0) \right \},\hspace{1.2cm}\\
Y^{[1]}=\left \{  (1,0), (0,1)    \right \},\hspace{1cm}\\
Y^{[2]}=\left \{   (2,0), (1,1), (0,2)   \right \},
\end{aligned}
\end{equation}
and apply Algorithm 9.4.1 in \cite{conn2000trust} on 
\begin{equation}
\beta=\left \{ 1, x_1, x_2, x_1^2, x_1x_2, x_2^2\right \}
\end{equation}
to find the Newton fundamental polynomials
\begin{equation}
\begin{aligned}
N_1^{[0]}(x_1,x_2)=1,\hspace{5cm}\\
N_1^{[1]}(x_1,x_2)=x_1, \hspace{1cm}N_2^{[1]}(x_1,x_2)=x_2,\hspace{3cm}\\
N_1^{[2]}(x_1,x_2)=\frac{1}{2}(x_1^2-x_1),\hspace{.3cm} N_2^{[2]}(x_1,x_2)=x_1x_2, \hspace{.3cm} N_3^{[2]}(x_1,x_2)=\frac{1}{2}(x_2^2-x_2).\\
\end{aligned}
\end{equation}

By Theorem 9.4.1 in \cite{conn2000trust}, we have 
\begin{equation}
\begin{aligned}
\lambda_0 (x)=\left(x_1-2\right)^4+\left(x_2-1\right)^3+e^{x_1+x_2},\hspace{3.2cm}\\
\lambda_1 (x)=\left(x_1-2\right)^4+\left(x_2-1\right)^3+e^{x_1+x_2} - 16,\hspace{3cm}\\
\lambda_2 (x)=\left[\left(x_1-2\right)^4+\left(x_2-1\right)^3+e^{x_1+x_2} - 16\right]- \left[ \left(e-14 \right)x_1+\left(e+1 \right)x_2 \right]. \\
\end{aligned}
\end{equation}
By (\ref{may22222}),
\begin{equation}\label{}
\begin{aligned}
m(x)=&16+\left(e-16 \right)x_1+e x_2+\left(\frac{e^2-2e+11}{2}\right)\left(x_1^2-x_1\right)+\\
&\left(e^2-2e-2\right)\left(x_1x_2\right)+\left(\frac{e^2-2e+1}{2}\right)\left(x_2^2-x_2\right).
\end{aligned}
\end{equation}
Figure \ref{figg:may222} illustrates the obtained model and its level curves. Figure \ref{figg:may223} compares the objective function and its model level curves over $\left[-2,2\right] \times \left[-2,2\right]$.
\begin{figure}[h!]
\centering
\includegraphics[width=\textwidth]{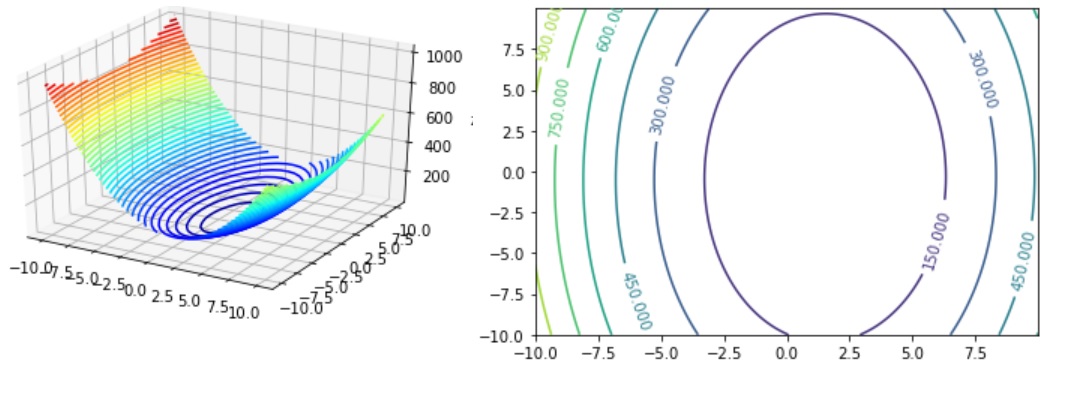}
\caption{\label{figg:may222} The quadratic model of $f(x_1,x_2)=\left(x_1-2\right)^4+\left(x_2-1\right)^3+e^{x_1+x_2}$ that is obtained by the Newton fundamental polynomials on the interpolation set $Y=\left \{ (0,0), (1,0), (0,1), (2,0), (1,1), (0,2) \right \}$.}
\end{figure}

\begin{figure}[h!]
\centering
\includegraphics[width=\textwidth]{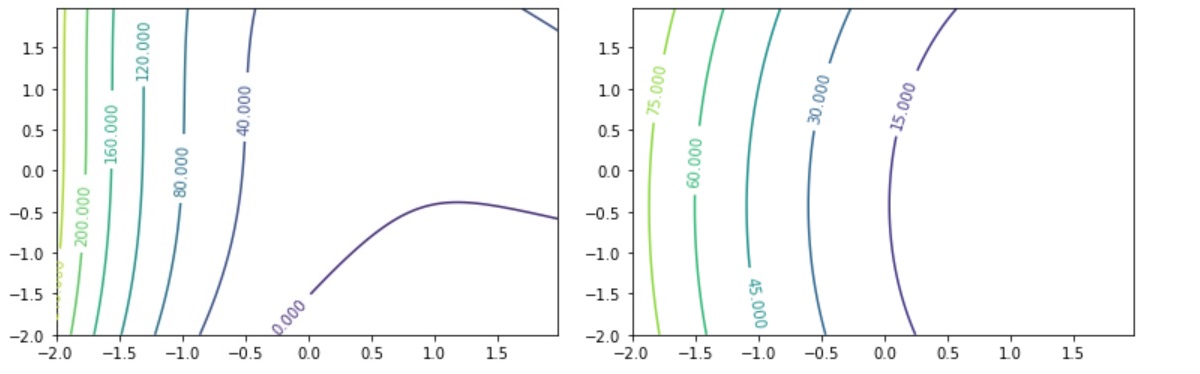}
\caption{\label{figg:may223} The level curves of $f(x_1,x_2)=\left(x_1-2\right)^4+\left(x_2-1\right)^3+e^{x_1+x_2}$ and its Newton model over $\left[-2,2\right] \times \left[-2,2\right]$}
\end{figure}

In practice, the threshold pivoting strategy is used to prevent $|N_ i^{[l]}(y^{j,[l+1])})|$ from becoming very large. 
At iterate $x^k$, the interpolation set $Y$ is said to be inadequate when $|D(Y)|$ exceeds a preassigned threshold by replacing $y_{\Delta}^i\in Y$ by any point $y$ inside the trust-region \cite{conn1997recent}, where
\begin{equation}
\begin{aligned}
\left \| x^k -y_{\Delta}^i\right \| \leq \Delta_k.
\end{aligned}
\end{equation}

If $\rho_k \not\geq \eta_1$, then we can attribute this result to one of two causes: either the interpolation set is inadequate or the trust region radius is too large. If the interpolation set $Y$ is inadequate, then a point $y^- \in Y$ is selected as 
\begin{equation}\label{ap121}
y^- =\text{arg} \max\limits_{y^i \in Y}  \left|L(y^i, y_{pr}^i)  \right |,
\end{equation}
where $y_{pr}^i$ is the potential replacement  of $y^i$ for all $y^i \in Y$, and is defined by
\begin{equation}\label{ap122}
y_{pr}^i=\text{arg} \max\limits_{\|y-x^k\| \leq \Delta_k} \left |L(y^i, y) \right |.
\end{equation}
If $\rho_k< \eta_1$ and the interpolation set $Y$ is adequate at $x^k$, then the trust-region radius should be reduced as
\begin{equation}
\Delta_{k+1}=\gamma_1 \Delta_k,
\end{equation}    
where $0<\gamma_1 < 1$. In fact, if the condition number of the system (\ref{MMarch21}) must be kept small, and the system (\ref{MMarch21}) must be as far away from singular as possible \cite{nocedal2006numerical}. 

The quadratic model assumption is computationally expensive, even if the model is updated and constructed based on the previous iteration model rather than constructing it from the scratch. It turns out that the number of required operations to update a quadratic model and calculate the corresponding step $s^k$ at every iteration is $O(n^4)$. The cost of each iteration might be reduced to $O(n^3)$ if the quadratic model is replaced by a linear model, which needs only $n+1$ interpolation points.

A neural network, which is made from different layers, is a beneficial tool for interpolation and function approximation. We now review the fundamental concepts of neural-networks, and in the next section we utilize this tool to propose a new model-based derivative-free trust-region method. A neural network contains an input layer, where the information enters the neural network; an output layer, where we can get the result out of the network, and a few hidden layers in between \cite{bishop2006pattern}. There are different types of Neural Network such as Feedforward Neural Network, Recurrent Neural Network which is basically used in Long Short Term Memory (LSTM) projects, Convolutional Neural Network, Radial basis function Neural Network, etc. In a general feed-forward network, each neuron activation $a_j$ is computed as a weighted sum of its inputs from the previous layer and it is then transformed by a activation function to returns $z_j$ as the output of the neuron. For instance if the input layer has $n$ variables $x_1, x_2, ..., x_{n}$, the hidden layer $j$ has $m_j$ neurons and the output layer has only one neuron, then the neuron ${k_1}$ in the first hidden layer is evaluated by
\begin{equation} \label{eq3}
z_{k_1}^{(1)}=\sigma_1(a_{k_1}^{(1)}),
\end{equation}
where $\sigma_1$ is a differentiable non linear function, which is called activation function,
\begin{equation} \label{eq2}
a_{k_1}^{(1)}= \sum_{i=1}^{{n}}{\omega_{{k_1}i}^{(1)} x_i},
\end{equation}
for ${k_1}=1,\dots, m_1$ and ${\omega_{{k_1}i}^{(1)}}$ is the weight from $x_i$ to neuron ${k_1}$ in the first hidden layer. If the neural-network contains $\ell$ hidden layers ($\ell+2$ layers), where hidden layer $\ell$ employs the differentiable activation function $h^{(\ell)}$, then the output of a feed forward neural-network is as follows:  
 \begin{equation}\label{eqqq4}
{NN}(W|x	)=h^{(\ell+1)} \left (\sum_{{k_{\ell}}=1}^{m_{\ell}} \left ( \omega_{1{k_{\ell}}}^{({\ell}+1)}h^{({\ell})}\left(...\left(\sum_{{k_2}=1}^{m_{2}} \omega _{{k_3}{k_2}}^{(3)}h^{(2)}\left ( \sum_{{k_1}=1}^{m_1}\left (\omega _{{k_2}{k_1}}^{(2)}h^{(1)}\left (\sum_{{k_0}=1}^{{m_0=n}}{\omega_{{k_1}{k_0}}^{(1)} x_{k_0}}\right)\right)\right)\right) \dots \right)\right) \right),
\end{equation}
where $W=(\omega_{11}^{(1)}, \dots, \omega_{m_{\ell} 1}^{({\ell}+1)})$, and for $r \geq 2$, $\omega_{{{\ell}}q}^{(r)}$ is the weight from $z_{q}^{(r-1)}$ to neuron ${{\ell}}$ in the $r^{\text{th}}$ hidden layer. We might select $h^{({\ell}+1)}$ to be identity function $I(x)=x$. Notice that (\ref{eq3}) is a ridge function, which is a function with the form as follows:
\begin{equation}
g(x,a,b)=\sigma (a^T x+b),
\end{equation}
where $\sigma$ is a nonlinear function, $a$ is the direction vector and $b$ is the bias. In this paper, all employed activation functions are continuously differentiable and can be considered as ridge functions.

A neural-network is a powerful tool for classification and regression. It can be employed to do an interpolation on a set $Y=\{y^1, y^2, \dots, y^p\}$ through minimization
of an appropriate loss function. Note that $y^j \in \mathbb{R}^n$ has $n$ components $y_1^j, y_2^j, \dots, y_n^j$. We might go through this regression procedure in two different approaches. The number of hidden layers and nodes of a neural-network are the hyper-parameters of the model, which means we must determine them in the beginning. The parameters in a neural-network that must be determined are the weights. The initial weights are often chosen arbitrarily.

Given the interpolation set $Y=\{y^1, y^2, \dots, y^p\}$, we might use all interpolation data points to find an loss function such as MSE. Then the loss function gets minimized to find the optimal weights.
\[
\min\limits_{W}  L(W)=L\left[\sum_{t=1}^{p} \left|f\left(y^{t}\right)-NN\left(W|y^{t}\right)\right|^q \right]^{\frac{1}{q}},
\]
where $y^{t} \in \mathbb{R}^n$ for all $t=1, 2, \dots, p$. However, one of the most important features of a good model is how well the trained model generalizes to new data. In other words, a neural-network is capable of making the error zero on the interpolation set, which means after the model is trained, all of the interpolation points lie on the model curve or surface. But, if we replace the current data point $y^-$ with a new data point $y^+$, then the trained model is no longer valid for the new interpolation set. Thus, a new model must be trained from scratch, which is often computationally expensive. So, generalization is one of the most important feature of a good model because the data that we sample is often incomplete and noisy. We aim to find a robust
model that does not overfit the interpolation set.

To help avoid overfitting, we split the interpolation set into two subsets, a training set $Y_{\text{Tr}}$ and a test set $Y_{\text{Te}}$. We train the model on $Y_{\text{Tr}}$ while $Y_{\text{Te}}$ is held back from the algorithm. After we have found the optimal weights of the neural-network on the training data set, we evaluate the trained model on the test set $Y_{\text{Te}}$ to find out how good the model might perform on unseen data points \cite{bishop2006pattern}.

The Universal Approximation Theorem (UAT), proven by Hornik \cite{hornik1991approximation}, shows that any continuous function $f \in C(\mathbb{R}^n)$ can be approximated with as few as a single hidden layer neural network under certain conditions with the input layer as the layer of random variables. Let $f:\Omega \subset \mathbb{R}^n \to \mathbb{R}$ be the function that we wish to approximate, and $m^{NN}$ be a model of $f$ that trained by a neural-network on $\Omega$ , where $\Omega$ is a compact subset of $\mathbb{R}^n$. The accuracy of approximation depends on how to measure closeness between a function and its corresponding model. The closeness is usually measured by the uniform distance between $f$ and $m^{NN}$ on domain $\Omega$:
\begin{equation}
\rho_{\mu,\Omega}=\sup\limits_{x \in \Omega} \left |f(x)-m^{NN}(x) \right |.
\end{equation}
The average performance with respect to the input environment measure $\mu$, where $\mu(R^k)<\infty$, is given as
\begin{equation}\label{m83}
\rho_{q,\mu}(f,m^{NN})=\left [\int_{R^k} \left |f(x)-m^{NN}(x) \right |^q d\mu(x) \right ]^{\frac{1}{q}},
\end{equation} and the choice corresponding to $q=2$, mean square error, is the most common used to measure the accuracy of the model \cite{hornik1991approximation}. Hashem et al. \cite{hashem1995improving} discussed how to improve the model accuracy by combining a set of trained neural networks if it is needed.

So a $\ell=1$ neural-network using sigmoidal activation functions can approximate a continuous function of a compact set in $\mathbb{R}^n$.  However, a $\ell=0$ is not capable of approximating a nonlinear continuous functions. A sigmoid function $\sigma(z)$ in $\mathbb{R}$ has the following properties:

\begin{equation}\label{m84}
\begin{aligned}
  \lim_{z \to -\infty}\sigma(z) &= 0, \\
    \lim_{z \to \infty}\sigma(z) &= 1, 
\end{aligned}
\end{equation}
and is defined as
\begin{equation}
\sigma(z)=\frac{1}{1- e^{-z}}.
\end{equation}
We approximate $f$ over $\Omega$ by splitting the domain $\Omega$ into a set of hypercubes $I_1, I_2, \dots, I_{\upsilon}$ and calculating
\begin{equation}\label{m71}
\hat{f}=\sum_{i=1}^{{\upsilon}}f(c_i)S_{I_i}(x),
\end{equation}
where $c_i$ is the center of $I_i$ and 
\begin{equation}\label{ap221}
S_{I_i}(x)=
\begin{cases}
                                   1  & \text{if } x\in I_i, \\
                                 0 & \text{otherwise}. 
 \end{cases}
\end{equation}
If $I_i=[a^i_1,b^i_1] \times [a^i_2,b^i_2] \times \dots \times [a^i_n,b^i_n]  \subset \mathbb{R}^n$ for $i=1, 2, \dots, \upsilon$, and 
\begin{equation}\label{ap222}
s(x)=
\begin{cases}
                                   1  & \text{if } x\geq 0, \\
                                 0 & \text{otherwise},
 \end{cases}
\end{equation}
then 
\begin{equation}
S_{I_i}(x)=S_{I_i}(x_1, x_2, \dots, x_n)=s \left( \left[\sum_{k=1}^{n}s\left(x^k-a^i_k \right)-s\left(x^k-b^i_k\right)\right]-n\right).
\end{equation}
We may express (\ref{m71}) in terms of a neural network with step activation function and two hidden layers. The input layer consists $x_1, x_2, \dots, x_n$, the first hidden layer consists of $2n$ neurons with step activation function, the second layer consists of $n$ linear activation function neurons, and one output neuron with step activation function. If $f$ is Lipschitz, then the error is $O(l)$, where the hypercubes $I_1, I_2, \dots, I_{\upsilon}$ have equal edges $l$ long \cite{zainuddin2008function}.

Therefore, we might approximate function $f$ on the compact set $\Omega$ by (\ref{m71}) where 
\begin{equation}
\Omega \approx \bigcup\limits_{i=1}^{{\upsilon}}I_i.
\end{equation}
Clearly, as $l \to 0$ and $\upsilon \to \infty$, then $\hat{f} \to f$. So a neural-network is capable of approximating any function on a compact subset, because if we define activation functions as (\ref{ap222}), then (\ref{ap221}) can be constructed by a net of connected neurons and $f$ can approximated by a neural-network. It turns out that the step function $s$ can be replaced by a sigmoid function $\sigma$ while the obtained result remains valid \cite{Geva}. Hecht et al. \cite{Hecht} prove that if $f: [0,1]^n \to \mathbb{R}$ is an arbitrary continuous function on $[0,1]^n=[0,1] \times [0,1] \times \dots \times [0,1]$, and $\sigma$ is the sigmoid function, then $f$ can be approximated by a three-layered (one hidden layer) feed forward neural-network. When we use sigmoid function, we are able to use its well-defined differentiability properties for using back-propagation and calculating the gradient and the Hessian of the model in this paper.

We can use the first and the second order derivatives obtained from a trained FNN to approximate the gradient of the model $g_k$ and the Hessian of the model $B_k$ in (\ref{e2}). If the activation functions of the trained FNNs are differentiable, the formulas to compute the first and the second order derivatives can be found in (\cite{hashem1992sensitivity}; \cite{bishop2006pattern}). The accuracy of a model obtained from an FNN can be improved by the multiresolution approach. That is we can simply add a few more neurons to the original neural network rather than building a completely new neural network \cite{csaji2001approximation}. Therefore, using neural network to approximate an appropriate model of the objective function within the trust-region (a compact subset of $\mathbb{R}^{n}$) may be computationally efficient. 

In the next section, we propose a new trust region method by employing deep neural network and using universal approximation theorem to maximum advantage. Throughout this paper, whenever we train a model by a neural network, we use differentiable activation and loss functions. Hence, the trained model $m_k^{NN}$ by a neural network is at least twice differentiable on $\mathbb{R}^{n}$. Moreover, the Hessian of the model can be estimated by back-propagation and remains bounded  on $\mathbb{R}^{n}$.

\section{ Neural Trust-region method}
In the rest of this paper, we consider an unconstrained optimization problem
\begin{equation}\label{apr293}
\min \limits_{x \in \mathbb{R}^n} f(x),
\end{equation}
where $f$ is locally Lipschitz continuous but it is possibly nonsmooth. We describe a new derivative-free trust-region algorithm in which a supervised machine learning technique is used to construct a robust model in trust-regions. Suppose at iteration $k$, the interpolation set $Y=\{y^1, y^2, \dots, y^p\}$ is given, and we are asked to find a good model of the objective function in
\begin{equation}\label{m81}
\Omega_k=\left \{x \in \mathbb{R}^{n} : \hspace{.2cm}\left \| x - x^k \right \| \leq \Delta_k \right \},
\end{equation}  
that satisfies (\ref{MMarch21}) with minimum error. We might use a classic interpolation method, such as direct methods or Lagrangian interpolation, to find a model satisfying (\ref{MMarch21}) for all $y^i \in Y$, then find the minimizer of the model in (\ref{m81}), and finally move to the next iterate. However, a low-order model e.g. linear or quadratic is often constructed because a high-order model has a high variance, which means if a new point $y^+$ is added to the interpolation set $Y$, then the current model does not remain a valid approximation of $f$ in the new trust-region and interpolation set, and the model must be trained from scratch. In the rest of this paper, we describe neural-network trust-region algorithms and address the aforementioned scenarios to some degree. 
\begin{assumption}\label{NTR assum1}
The objective function $f$ is Lipschitz continuous and bounded below on $\mathbb{R}^{n}$. 
\end{assumption}
\begin{assumption}\label{NTR assum2}
All activation functions and loss functions that are employed to build a model throgh a neural-network are at least twice differentiable. 
\end{assumption}
Assumption \ref{NTR assum1} guarantees that there exists a constant $\kappa_f$ such that for all $x \in \mathbb{R}^n$, $f(x) \geq \kappa_f.$ Assumption \ref{NTR assum2} guarantees that through backpropagation all derivatives of the trained model $\frac{\partial m_k}{\partial x_1}, \frac{\partial m_k}{\partial x_2}, \dots, \frac{\partial m_k}{\partial x_n}, \frac{\partial ^2 m_k}{\partial x_1^2}, \frac{\partial ^2 m_k}{\partial x_1\partial x_2} \dots$ are available. So the gradient and Hessian of the model can be constructed with a relatively insignificant cost automatically. 

Since the objective function might not be differentiable, the algorithm terminates at a Clarke stationary point \cite{Clarke}. Clarke generalized derivative of $f$ along direction $d$ is defined as 
\begin{equation}\label{apr291}
f^o (x;d)=\lim  \limits_{\substack{{y \to x}  \\ { \alpha \to 0}}} \text{sup}\frac{f(y+\alpha d)-f(y)}{\alpha},
\end{equation}
and the Clarke generalized gradient of $f$ at $x$ is defined as 
\begin{equation}\label{ap292}
\partial f(x)=\left  \{p\in \mathbb{R}^n : f^o (x;v) \geq v^Tp, \text{ for all } v \in \mathbb{R}^n \right\}.
\end{equation}
where $f$ is Lipschitz near $x$. From (\ref{apr291}) and (\ref{ap292}), we have
\begin{equation}
f^o (x;d)=\max \left \{d^Tp: p \in \partial f(x) \right \},
\end{equation}
and $x^*$ is said to be a Clarke stationary point for (\ref{apr293}) if 
\begin{equation}
f^o (x^*;d)\geq0,
\end{equation}
for all $d \in \mathbb{R}^n$ or, equivalently, if $0\in\partial f(x)$.In the following subsection, we describe how to employ a neural-network to solve an unconstrained minimization problem (\ref{apr293}).

\subsection{Neural trust-region using a quadratic model}\label{subsec:unsupervised}
Since the objective function $f$ derivative information are not available, we are not able to use Taylor-series theorem to construct a model of the function at each iteration. So, at every iterate $x^k$, a quadratic model $m^{NN}_k$ of the objective function within an appropriate trust-region centered $x^k$ is required to be constructed through a feed forward neural-network. We aim to make the most benefit from the neural-network backpropagation properties, and step toward solving the corresponding subproblem 
\begin{equation}
s^k=\arg \min\limits_{\| s \| \leq   \Delta_k} { m^{NN}_k,}
\end{equation}
through the neural-network rather than using Steihaug-Toint, GLTR, etc. The sensitivity analysis of neural-network models is thoroughly investigated in (\cite{hashem1992sensitivity}; \cite{davis1989sensitivity}), which can be used to see the effect of replacing a data point $y^-$ in the current interpolation set with a new point $y^+$ on the loss function and the corresponding subproblem solution. It turns out that a neural-network model is very robust when it comes to changing data points, especially if we change one at a time. It means, when we replace $y^-$ with $y^+$ in the interpolation set, the neural-network model is not required to be trained from scratch, which means it might be updated at every iterate with significantly lower cost.

We now describe how to construct a good quadratic model of the objective function and solve the corresponding subproblem in the current trust-region using a neural-network. Since for a given interpolation set $Y$, a neural-network model parameters are the weights of the connections, which we denoted as a weight matrix $W$, and the model as $m_k^{NN}(W|Y)$, $m_k^{NN}(W|x)$ or simply $m_k^{NN}(W)$. We need to find a neural-network weight matrix $W$ in order to build a quadratic model $m_k^{NN}$ that satisfies (\ref{MMarch21}) and is valid in
(\ref{m81}).
The trained model $m_k^{NN}$ is said to be valid in (\ref{m81}) if and only if for all points $x$ in the current trust-region and for some constant $\kappa$,
\begin{equation}\label{validity}
\left\|f(x)-m_k^{NN}(W^*,x)\right  \|\leq \kappa \Delta_k^2,
\end{equation}
where $W^*$ is the optimal weight matrix. Note that we may relax (\ref{MMarch21}) to some degree to allow the model $m_k^{NN}$ to satisfy (\ref{validity}). 

Since in this section, we only consider training a quadratic model, we only need to approximate the gradient and the Hessian of the objective function within the trust-region at every iteration. On the other hand, the gradient and the Hessian of an objective function can be written as the linear combination of their vector spaces bases
\begin{equation}
\begin{aligned}
  g &= \sum_{i=1}^{n} w^{g}_i e_i \\
  H &=\sum_{i<j} w^{H}_{ij} (E_{ij}+E_{ji})+\frac{1}{2}\sum_{i=1}^{n} w^{H}_{ii} E_{ii}, 
\end{aligned}
\end{equation}
where $e_i$ is the $i^{th}$ standard basis vector whose all its elements equal to zero except the $i^{\text{th}}$ that is equal to one, and $E_{ij}$ is the matrix standard basis matrix whose all its elements equal to zero except the $(i,j)$ that is equal to one.
A neural-trust-region method, which is a trust-region algorithm employing neural-network and deep learning to find a model within the trust region of every iterate, looks for the weights 
\[W=\left [w^{g}_1, w^{g}_2, \dots, w^{g}_n, w^{H}_{11}, w^{H}_{12}, \dots, w^{H}_{nn} \right]\]
to minimize mean-squared error,
\begin{equation}\label{m64}          
 L_1=MSE(W)=\bar{\zeta_1}\left[\frac{1}{p}\sum_{i=1}^{p}\left[m^{NN}_k\left(W,s|y_i)-f\left(y_i\right)\right]^2\right]^{\frac{1}{2}}\right]+\bar{\zeta_2} \left [\lambda_1-c \right]^2,
 \end{equation}
where 
\begin{equation}\label{e3}
m^{NN}_k (W,s|y_i ) = f(y^i) +\left (\sum_{i=1}^{n} w^{g}_i e_i \right )^T s +\dfrac{1}{2} s^T \left[\sum_{i<j} w^{H}_{ij} (E_{ij}+E_{ji})+\frac{1}{2}\sum_{i=1}^{n} w^{H}_{ii} E_{ii}\right]s,
\end{equation}
$\lambda_1$ is the smallest eigenvalue of $\nabla_{xx}m^{NN}_k$, which is available through backpropagation, $c>0$, $\bar{\zeta_1}\geq 0$ and $\bar{\zeta_2}\geq 0$. Note that we may choose  $\bar{\zeta_2}\geq\bar{\zeta_1}$ such that $\lambda_1>0$, which means that the trained quadratic model is convex. We may choose $ \bar{\zeta_2}=0$ and do not place any assumptions on the concavity of the model and leave it for trust-region boundary to enforce. But the loss function $L_1$ in (\ref{m64}) by itself is not a proper loss function for training the model because $s$ is unknown and has significant impact on model details. To alleviate this issue, we might define a new loss function $L_2$ for a neural-network that is constructed in series with first one. The first neural-network with loss $L_1$ is called the \textit{parent}, and the second neural-network with loss function $L_2$ is called \textit{child}. We call the whole neural-network as \textit{parent-child} net with a loss that is a linear combination of $L_1$ and $L_2$.

We now describe how to build the child neural-network with a proper loss function $L_2$. The trained model from the parent neural-network is $m_k^{NN}(s|W)$, which means the parameter $s$ is still unknown and should be determined. Note that we could minimize the parent trained model $m_k^{NN}(s|W)$ within the current trust-region by common methods such as Steihaug-Toint and GLTR if it is valid in the current trust-region to find $s$, which is the solution of the subproblem. But we can also use a child neural-network framework to find a step $s^k$ for the subproblem corresponding to $m_k^{NN}$ satisfying optimality conditions. A loss function might be simply defined to seek a pair $(s,w^*)$ satisfying the optimality conditions, 
\begin{equation}\label{m67}
 L_2=\zeta_1 \left [ \left(\nabla_{xx}m_k^{NN}+w^*I \right ) s+\nabla_x m_k^{NN} \right ]^2+\zeta_2 \left[w^*(\Delta_k-\|s\|)\right ]^2+\zeta_3 \left[\tilde{\lambda}_1-\tilde{c}\right]^2,
 \end{equation}
where $ \zeta_1\geq0,$ $ \zeta_2\geq0$ and $ \zeta_3\geq0$ are the weights for each term, $\tilde{\lambda}_1$ is the smallest eigenvalue of $\nabla_{xx}m_k^{NN}+w^*I$, $\tilde{c}\geq 0$ is a predetermined constant and 
\begin{equation}\label{m91}
               s=\sum_{i=1}^{n} w^{s}_i e_i.
 \end{equation}
Conn et al. \cite{CSV} (Theorem 10.1) showed that in a derivative-free trust-region method, 
\begin{equation}\label{emroozap1}
m_k(x^k)-m_k(x^k -t_k^C \nabla m_k(x^k)) \geq \frac{1}{2}\|\nabla m_k(x^k)\| \min \left\{\frac{\left\|\nabla m_k(x^k) \right \|}{\|\nabla_{xx} m_k(x^k)\|}, \Delta_k \right\},
\end{equation}
where 
\begin{equation}
t_k^C =\text{arg} \min \limits_{t>0}m_k \left(x^k-t \nabla m_k(x^k) \right), \text{ \hspace{.2cm} and \hspace{.2cm}} x^k-t \nabla m_k(x^k) \in \Omega_k.
\end{equation}
We aim to train a model whose corresponding subproblem solution $s^k$ satisfies 
\begin{equation}
m^{NN}_k(x^k)-m^{NN}_k(x^k+s^k)\geq c^{*} \left[m_k(x^k)-m_k(x^k -t_k^C \nabla m_k(x^k)) \right].
\end{equation}
for some $c^{*}>0$ and for all $k$. Hence, we might add another term to our loss function as 
\begin{equation}\label{NTR-Cauchy}
L_3=\left [\left [m^{NN}_k(x^k)-m^{NN}_k\left(x^k+s^k\right)\right] -c^{*} \left [m_k\left(x^k\right)-m_k\left(x^k -t_k^C \nabla m_k\left(x^k\right)\right)\right ]\right]^2-\tilde{c},
\end{equation}
for some ${\tilde{c}} \geq 0$.
Note that in (\ref{m67}) the second term satisfies the complementary condition and the third term satisfies the curvature condition. Therefore, we may build a parent-child net to find the step $s^k$ of the subproblem by minimizing 
 \begin{equation}\label{erm61}
             \text{Overall-Loss}(\hat{W})=\tilde{\zeta}_1 L_1+\tilde{\zeta}_2 L_2+\tilde{\zeta}_3 L_3,	
 \end{equation}
where 
\begin{equation}
\hat{W}=\left[w^{g}_1, w^{g}_2, \dots, w^{g}_n, w^{H}_{11}, w^{H}_{12}, \dots, w^{H}_{nn}, w^{s}_1, w^{s}_2, \dots, w^{s}_n, w^* \right],
\end{equation}
$\tilde{\zeta}_1\geq0, \tilde{\zeta}_2,\tilde{\zeta}_3\geq0$ are adjustable hyper-parameter for the loss function. Note that if the Hessian of a trained model in (\ref{m64}) turns out to be positive semidefinite, then the weight $w^*$ in (\ref{m67}) becomes zero. Moreover, $w^{s}_i$'s and $w^*$ are dependent on $w^{g}_i$'s and $w^{H}_{ij}$'s for $1\leq i\leq j\leq n$. 

We might assume the Cauchy reduction (\ref{NTR-Cauchy}) is always obtained through the neural-network trust-region algorithm and then convert the loss function (\ref{erm61}) to 
\begin{equation}\label{erm611}
             \text{Overall-Loss}(\hat{W})=\tilde{\zeta}_1 L_1+\tilde{\zeta}_2 L_2.
 \end{equation}

After the algorithm trains a valid model $m_k^{NN}$ and calculates the minimizer of the corresponding subproblem, it updates the iterate and the interpolation set. Since we initially assumed that a model of the objective function must be quadratic, the bias of the trained model is relatively high but the variance might be relatively low. 

\begin{algorithm}[H]
            \caption{Neural-trust-region algorithm based on a smooth quadratic model} 
            \label{alg:ALG1}
            \begin{flushleft}
                \textbf{Step 0: Initialization.} 
                 An initial point $x_0$ and an initial trust-region radius $\Delta_0>0$ are given. The constants  $0\leq \eta_1 \leq1$, $0<\gamma_1  <\gamma_2 \leq 1$ and $\epsilon>0$. Compute $f(x_0)$ and set $k=0$.\\
             \end{flushleft}    
             
             \begin{flushleft}
                \textbf{Step 1: Model definition and Step calculation.}  Given the interpolation set $Y=\{y^1,y^2, \dots, y^{p}\}$ within the current trust-region centered $x^k$ with radius $\Delta_k$, calculate step $s^k$ that sufficiently reduces the model by training a valid model $m_k^{NN}$ in the current trust-region and minimizing loss (\ref{erm61}) of a neural network.
                 
             \end{flushleft}     
             
             \begin{flushleft}
                \textbf{Step 2: Acceptance of trial point.} Compute $f(x^k+s^k)$ and 
                 $$\rho_k = \frac{f(x^k) - f( x^k + s^k )}{ m_k ( x^k )-m_k( x^k+s^k )}.$$
   \hspace{1.2cm}           If $\rho_k\geq \eta_1$, then define $x^{k+1}=x^k+s^k$; otherwise define $x^{k+1}=x^{k}$.

             \end{flushleft}    
             
             \begin{flushleft}
                \textbf{Step 3: Trust-region radius and training set update (see \cite{conn1997convergence} and \cite{conn1997recent}).} 
                
  \hspace{1.2cm}           If $\rho_k\geq \eta_1$, then $\Delta_{k+1}=\gamma_2 \Delta_k$, select the exiting point $y^-$ by 
 \begin{equation}\label{y-}
y^-  =\text{arg} \max\limits_{y^i \in Y} L(x^+, y^i),
\end{equation}
 \hspace{1.4cm} where $L(.,.)$ is the Lagrangian function defined in (\ref{ap112}), and replace $y^- \in Y$ with $x^+=x^k+s^k$. 
 
  \hspace{1.2cm}          If $\rho_k< \eta_1$ \textit{and} the interpolation set $Y$ is inadequate, then for every $y^i \in Y$ find $y_{pr}^i$ by (\ref{ap122}). Then select the 
 
 \hspace{1.2cm}  exiting point $y^-$ by (\ref{ap121}). Finally, replace $y^-$ with a point $y^+$ in the trust-region to improve $D(Y)$ in (\ref{ap123}).
 
  \hspace{1.2cm}           If $\rho_k< \eta_1$ \textit{but} the interpolation set $Y$ is adequate, then $\Delta_{k+1}=\gamma_1 \Delta_k$.

\hspace{.2cm}Increment $k$ by $1$ and go to Step 1.
             \end{flushleft}        
    \end{algorithm}
The loss functions $L_1$ and $L_2$ in (\ref{m64}) and (\ref{m67}) might be replaced by more efficient loss functions. The interpolation set $Y$ might be sampled in a way that we obtain a balanced loss function on the boundary points and inner points of the trust-region while it is kept to be adequate. In this case the model represents the objective function more appropriately on the boundary of the trust-region.

\subsection{Convergence analysis of Algorithm \ref{alg:ALG1}}
According to the universal approximation theorem, any continuous function $f$ can be approximated by a $\ell \geq 1$ feed-forward neural network as good as we wish. In other words, for a given $\epsilon >0$, there exists $h \in \mathbb{N}$ hidden neurons such that 
\begin{equation}
\|f(x)-m_k^{NN}(x)\|_2\leq \epsilon,
\end{equation}
where $m_k^{NN}$ is the trained neural-network model with $M$ hidden neurons. So we can have the following assumption throughout this paper, which means the neural-network model is valid.
\begin{assumption}\label{NTR:assum3}
For a given model $m_k$, there exists $\kappa>0$ such that for all points $x\in\Omega_k$ and for all $k$, 
\begin{equation}
\left \|f(x)-m_k^{NN}(x) \right \|_2\leq \kappa \Delta_k^2.
\end{equation}
\end{assumption}

\begin{lemma}\label{NTR:lemma1}
For all $k$, the trained model $m_k^{NN}(x)$ in trust-region (\ref{m81}) is at least twice differentiable with respect to $x$.
\end{lemma}
\begin{proof}
By Assumption \ref{NTR assum2}, all discussed activation functions and loss functions that neural-trust-region algorithms use are at least twice continuously differentiable. Moreover, the model $m_k^{NN}$ is a composition of continuously differentiable functions similar to (\ref{eqqq4}). Thus, the model $m_k^{NN}$ is at least twice continuously differentiable and its derivatives can be obtained automatically through back propagation.

\end{proof}

\begin{lemma}\label{NTR:lemma2}
For all $k$, the Hessian of the trained model $m_k^{NN}$ in trust-region (\ref{m81}) remains bounded; that is for all $x$ in the current trust-region and for some constant $\kappa_{hm}$, 
\begin{equation*}
\left\| \nabla_{xx} m_k^{NN} \right \| \leq \kappa_{hm}.
\end{equation*}

\end{lemma}
\begin{proof}
Without loss of generality, suppose we train a model $m_k^{NN}$ of $f:\mathbb{R}^n \to \mathbb{R}$ by employing a neural-network with one hidden layer and $h$ hidden neurons on the interpolation set $Y=\{ y^1, y^2, \dots, y^p\}$. Hence, the trained model has the form
\begin{equation}
m_{k}^{NN}(x_1,x_2, \dots, x_n)=\sum_{j=1}^{h} w_j^{(2)}\sigma(\sum_{i=1}^{n}w_{ij}^{(1)}x_i),
\end{equation}
where $x=(x_1,x_2, \dots, x_n)$, $\sigma$ is the sigmoid function, $w_{ij}^{(1)}$ is the weight of the connection that goes form the input $x_i$ to the $j^{th}$ hidden neuron, and $w_j^{(2)}$ is the weight of the connection that goes to the output neuron form the $j^{th}$ neuron. Through back propagation, the model derivatives
\begin{equation}\label{senkr}
\frac{\partial{m_{k}^{NN}}}{\partial x_r}=\sum_{j=1}^{h} w_j^{(2)}w_{rj}^{(1)} \sigma(\sum_{i=1}^{n}w_{ij}^{(1)}x_i) -w_j^{(2)}w_{rj}^{(1)} (\sigma(\sum_{i=1}^{n}w_{ij}^{(1)}x_i))^2,
\end{equation}
and
\begin{equation}
\begin{aligned}
\frac{\partial^2{m_{k}^{NN}}}{\partial x_r \partial x_s}=&\sum_{j=1}^{h} \left[w_j^{(2)}w_{rj}^{(1)}w_{sj}^{(1)} \sigma(\sum_{i=1}^{n}w_{ij}^{(1)}x_i)(1-\sigma(\sum_{i=1}^{n}w_{ij}^{(1)}x_i)) \right]-\\
& 2\sum_{j=1}^{h} \left[w_j^{(2)}w_{rj}^{(1)}w_{sj}^{(1)} (\sigma(\sum_{i=1}^{n}w_{ij}^{(1)}x_i))^2(1-\sigma(\sum_{i=1}^{n}w_{ij}^{(1)}x_i)) \right],
\end{aligned}
\end{equation}
for all $1 \leq r \leq s \leq n$, can be obtained with relatively low cost. Due to Assumption \ref{NTR assum1}, the structure of a neural-network and the definition of loss function, we know $w_{ij}^{(\ell)}<\infty$ for all $i=1, 2, \dots, n$, $j=1, 2, \dots, h$ and $ \ell=1, 2$. On the other hand, since $0< \sigma(x) <1$ and $h< \infty$, there exists a constant $\kappa_{rs}$ such that
\begin{equation}\label{ap241}
\left|\frac{\partial^2{m_{k}^{NN}}}{\partial x_r \partial x_s}\right| <\kappa_{rs}.
\end{equation}
Thus, (\ref{ap241}) simply implies that there exists a constant $\kappa_{hm}$ such that 
\begin{equation*}
\left\| \nabla_{xx} m_k^{NN} \right \| \leq \kappa_{hm}.
\end{equation*}

\end{proof}

\begin{lemma}\label{lemma_upperbound_for_sd}
Suppose the loss function given by (\ref{erm61}) is employed to train a neural-network model $m_k^{NN}$. Then for all $k$ and for some $\Lambda \geq 0$, the subproblem solution satisfies
\begin{equation}
(s^k)^T\nabla_x m_k^{NN}(x^k) \leq -\Lambda \|s^k\|^2.
\end{equation}
\end{lemma}
\begin{proof}
Neural Network methods guarantee that we are able to train a model $m_k^{NN}$ such that the corresponding subproblem solution $s^k$ results in zero
residual $L_2$ (\ref{m67}). So there exists a constant $w^*\geq $ such that 
\begin{center}
$(\nabla_{xx}m_k^{NN}+w^*I)s^k+\nabla_{x}m_k^{NN}=0,$ \\
$\nabla_{xx}m_k^{NN}+w^*I$ \text{ is positive semidefinite, and}\\
$w^*(\Delta_k-\|s^k\|)=0. $\hspace{1cm}
\end{center}
Note that $\nabla_{xx}m_k^{NN}+w^*I$ positive semidefinite because the process builds $\nabla_{xx}m_k^{NN}$ as positive semidefinite.
If $w^*=0$, then 
\begin{equation}
\nabla_{xx}m_k^{NN}(x^k)s^k=-\nabla_{x}m_k^{NN}(x^k),
\end{equation}
and  since $\nabla_{xx}m_k^{NN}$ is positive semidefinite, we have 

\begin{equation}
(s^k)^T\nabla_{x}m_k^{NN}(x^k)=-(s^k)^T\nabla_{xx}m_k^{NN}(x^k)s^k \leq -\lambda_1 \|s^k\|^2,
\end{equation}
where $\lambda_1$ is the smallest egienvalue of $\nabla_{x}m_k^{NN}(x^k)$.

If $w^*> 0$, then
\begin{center}
$\Delta_k=\|s^k\|,$\\
$(\nabla_{xx}m_k^{NN}(x^k)+w^*I)s^k=-\nabla_{x}m_k^{NN},$
\end{center}
and 
\begin{equation}
(s^k)^T\nabla_{x}m_k^{NN}(x^k)=-(s^k)^T(\nabla_{xx}m_k^{NN}+w^*I)(x^k)s^k \leq -\tilde{\lambda}_1 \|s^k\|^2,
\end{equation}
where $\tilde{\lambda}_1$ is the smallest eigenvalue of $(\nabla_{xx}m_k^{NN}+w^*I)(x^k)$. Letting $\Lambda= \min \{{\lambda}_1,\tilde{\lambda}_1\}$, then for all $k$, we have
\begin{equation}
(s^k)^T\nabla_{x}m_k^{NN}(x^k) \leq -\Lambda \|s^k\|^2.
\end{equation}
\end{proof}

\begin{theorem} \label{NTR-Theorem1}
Suppose a quadratic model and loss function, found at each iteration of Algorithm 1, are used to train a $\ell\geq 1$ neural-network model.  Then for all $k$, there exists $\beta >0$ so that
\begin{equation}
m^{NN}_k(x^k)-m^{NN}_k(x^k+s^k)\geq \beta \|s^k\|^2
\end{equation}
\end{theorem}
\begin{proof}

Without loss of generality, suppose a $\ell=1$ neural network model with and $h$ hidden neurons is employed,
\begin{equation}
m_{k}^{NN}(x)=\sum_{j=1}^{h} w_j^{(2)}\sigma \left(\sum_{i=1}^{n}w_{ij}^{(1)}x_i \right),
\end{equation}
where $x=(x_1,x_2, \dots, x_n)$, $\sigma$ is sigmoid function. So
\begin{equation}
m_k^{NN}(x^k)-m_k^{NN}(x^k+s^k)=-(s^k)^T\nabla_x m_k^{NN}(x^k)-\frac{1}{2}(s^k)^T \left (\nabla_{xx}m_k^{NN}+w^*I \right)(x^k)(s^k),
\end{equation}
where $w^*$ is determined by the neural-network to minimize the loss function. By the definition of the $L_2$ in (\ref{erm61}), we have
\begin{equation}
m_k^{NN}(x^k)-m_k^{NN}(x^k+s^k)=-(s^k)^T\nabla_x m_k^{NN}(x^k)+\frac{1}{2}(s^k)^T\nabla_x m_k^{NN}(x^k),
\end{equation}
so
\begin{equation}
m_k^{NN}(x^k)-m_k^{NN}(x^k+s^k)=-\frac{1}{2}(s^k)^T\nabla_x m_k^{NN}(x^k).
\end{equation}
By Lemma \ref{lemma_upperbound_for_sd}, there exists a constant $\beta \in (0,1)$ such that
\begin{equation}\label{ap261}
m_k^{NN}(x^k)-m_k^{NN}(x^k+s^k)\geq \beta \|s^k\|^2.
\end{equation}

\end{proof}
\begin{corollary}
If the neural-network model finds nonzero vector weights $w^s=(w_1^s, w_2^s, \dots, w_n^s) \neq 0$, then 
\begin{equation}
m_k^{NN}(x^k+s^k)<m_k^{NN}(x^k).
\end{equation}
\end{corollary}
\begin{proof}
Clearly if $w^s=(w_1^s, w_2^s, \dots, w_n^s) \neq 0$, then $s^k=\sum_{i=1}^n w_i^se_i \neq 0$, and from (\ref{ap261}), the desired result can be obtained.
\end{proof}

\begin{corollary} 
In Algorithm \ref{alg:ALG1}, where the quadratic model and the loss function given by (\ref{erm611}) are employed to train a neural-network model $m_k^{NN}$, for all $k$ and for some $\hat{c}>0$ the subproblem solution satisfies 
\begin{equation}\label{emrooz2app}
\left[m^{NN}_k(x^k)-m^{NN}_k(x^k+s^k)\right]\geq \hat{c}\left[m_k(x^k)-m_k(x^k +s^C)\right],
\end{equation}
where $s^C=-t_k^C \nabla m_k(x^k)$ is the Cauchy step.
\end{corollary}
\begin{proof}
By Theorem \ref{NTR-Theorem1} we have
\begin{equation}
m_k^{NN}(x^k)-m_k^{NN}(x^k+s^k)\geq \beta \|s^k\|^2,
\end{equation}
and since $\|s^k\| \leq \Delta_k$, which means $\|s^k\| ={\bar{c}} \Delta_k$ for some constant $\bar{c} \in (0,1]$, we have 
\begin{equation}
m_k^{NN}(x^k)-m_k^{NN}(x^k+s^k)\geq \beta \bar{c}^2 \Delta_k^2,
\end{equation}
which means the model reduction is of order $\Delta_k^2$. On the other hand, without loss of generality, if we assume $\Delta_k \approx \|\nabla m_k(x^k)\|$, then the the model decrease attained by the Cauchy step in (\ref{emroozap1}) is of order $\Delta_k^2$. Therefore, there exists $\hat{c} >0$ that satisfies (\ref{emrooz2app}).
\end{proof}

\begin{lemma}\label{radiustozero}
Any sequence $\{ \Delta_k\}$ produced by Algorithm \ref{alg:ALG1} satisfies
\begin{equation}\label{deltazero}
\lim\limits_{k \to \infty} \Delta_k=0.
\end{equation}
\end{lemma}
\begin{proof}
By contradiction, suppose $\lim \limits_{k \to \infty} \Delta_k \neq 0$, which means there exist $\epsilon >0$ and a subsequence $\{ \Delta_{k_t}\}_{t=1}^{\infty}$ such that $\Delta_{k_t}> \epsilon$ for all $t=1,2, \dots$, which means there exists an infinite number of iterations for which $\Delta_k > \epsilon$. Thus, Step 3 in Algorithm \ref{alg:ALG1} indicates that there exist infinite number of successful iterations, at which $\rho_k \geq \eta_1$. By the definition of $\rho_k$ and Theorem \ref{NTR-Theorem1}, for a successful iterate, we have
\begin{equation}
f(x_k)-f(x_{k+1}) \geq \eta_1 \beta \|s^k\|^2.
\end{equation}
Suppose $S=\{k_s | \rho_{k_s} \geq \eta_1 \text{  for } k_s \in \{0, 1, \dots\}\}$ be the set of successful indices, and $k_i$ is the $i^{\text{th}}$ successful iterate index. So
\begin{equation}
f(x_{k_1})-f(x_{k_{i}}) \geq \eta_1 \beta \sum_{q=1}^{i} \|s^{k_q}\|^2,
\end{equation}
which means
\begin{equation}
\lim \limits_{i \to \infty}f(x_{k_1})-f(x_{k_{i}}) =\infty,
\end{equation}
that contradicts Assumption \ref{NTR assum1}. Hence, (\ref{deltazero}) is true.
\end{proof}

In the following theorems, we show that Algorithm \ref{alg:ALG1} converges to a Clarke stationary point. The proofs for these theorems are similar to what can be found in \cite{Audet} and \cite{Liuzzi}.

\begin{theorem}\label{NTR-Theorem2}
The neural-trust-region algorithm \ref{alg:ALG1} terminates in a finite number of iterations, or generates an infinite sequence $\{(x^k,s^k)\}_{k=1}^{\infty}$ such that 
\begin{equation}
\lim \limits_{k \to \infty}f^o(x^k, \frac{s^k}{\|s^k\|})\geq 0.
\end{equation}
\end{theorem}

\begin{proof}
Without loss of generality, let $\{\Delta_k\}$ be the sequence of trust-region radii such that
\begin{center}
$\lim \limits_{k \to \infty}\Delta_k=\Delta^*.$
\end{center}
By Lemma \ref{radiustozero}, if Algorithm \ref{alg:ALG1} does not terminate in a finite number of iterations, then $\Delta^*=0$. Let $U=\{k: \rho_k < \eta_1\}$ be the set of unsuccessful indices, at which $\Delta_{k+1} <\Delta_k$. So for each $k \in U$,
\begin{equation}
\frac{f(x^k) -f(x^k+s^k)}{m_k^{NN}(x^k)-m_k^{NN}(x^k+s^k) }< \eta_1,
\end{equation}
so
\begin{equation}
f(x^k+s^k)-f(x^k) >-\eta_1 \left [m_k^{NN}(x^k)-m_k^{NN}(x^k+s^k)\right ],
\end{equation}
and by Theorem \ref{NTR-Theorem1},
\begin{equation}
f(x^k+s^k)-f(x^k) >-\eta_1 \beta \|s^k\|^2.
\end{equation}
Hence,
\begin{equation}
\frac{f(x^k+s^k)-f(x^k)}{\|s^k\|} >-\eta_1 \beta \|s^k\|,
\end{equation}
which is 
\begin{equation}
\frac{f(x^k+\|s^k\|( \frac{s^k}{\|s^k\|}))-f(x^k)}{\|s^k\|} >-\eta_1 \beta \|s^k\|,
\end{equation}
and
\begin{equation}
\lim \limits_{k \to \infty}\frac{f\left(x^k+\|s^k\|( \frac{s^k}{\|s^k\|})\right)-f(x^k)}{\|s^k\|} >\lim \limits_{k \to \infty} -\eta_1 \beta \|s^k\|.
\end{equation}
But if $\Delta_k \to 0$, then $\|s^k\| \to 0$, which means 
\begin{equation}
\lim \limits_{\substack{{x^k \to x^*}  \\ { \alpha \to 0^+}}}\frac{f\left(x^k+\alpha( \frac{s^k}{\|s^k\|})\right)-f(x^k)}{\alpha} >\lim \limits_{\alpha \to 0^+} -\eta_1 \beta \alpha,
\end{equation}
which is
\begin{equation}
\lim \limits_{\substack{{x^k \to x^*}  \\ { \alpha \to 0^+}}} \text{sup}\frac{f\left(x^k+\alpha( \frac{s^k}{\|s^k\|})\right)-f(x^k)}{\alpha}=\lim \limits_{k \to \infty}f^o(x^k, \frac{s^k}{\|s^k\|}) \geq 0.
\end{equation}
\end{proof}

\begin{theorem} \label{NTR-Theorem3}
The neural-trust-region algorithm \ref{alg:ALG1} terminates in a finite number of iterations, or generates an infinite sequence $\{x^k\}$ which satisfies
\begin{equation*}
\lim_{k\to\infty} x^k=x^*,
\end{equation*}
where $f^o(x^*; d) \geq 0$ for all $d \in \mathbb{R}^n$.
\end{theorem}
\begin{proof}
If $\tilde{d}=\frac{s^*}{\|s^*\|}$, then by Theorem \ref{NTR-Theorem2}, the desired result can be obtained. By contradiction, we suppose that there exists a unit vector $\tilde{d} \neq \frac{s^*}{\|s^*\|} \in \mathbb{R}^n$ such that
\begin{equation}
f^o(x^*; \tilde{d}) < 0,
\end{equation}
which means 
\begin{equation}
f^o(x^*, \tilde{d}) =\lim \limits_{\substack{{x^k \to x^*}  \\ { \alpha \to 0^+}}} \text{sup}\frac{f(x^k+\alpha \tilde{d})-f(x^k)}{\alpha} < 0.
\end{equation}
On the other hand, by Theorem \ref{NTR-Theorem2},
\begin{equation}
f^o(x^*, \frac{s^*}{\|s^*\|}) =\lim \limits_{\substack{ { k\to \infty}\\ {x^k \to x^*}  \\ { \alpha \to 0^+}}} \text{sup}\frac{f\left(x^k+\alpha( \frac{s^k}{\|s^k\|})\right)-f(x^k)}{\alpha} \geq 0,
\end{equation}
which means 
\begin{equation}
\lim \limits_{\substack{{ k\to \infty}\\{x^k \to x^*}  \\ { \alpha \to 0^+} \\ { s^k \to s^*}}} \text{sup}\frac{f\left(x^k+\alpha( \frac{s^k}{\|s^k\|})\right) - f(x^k+\alpha \tilde{d})}{\alpha}>0.
\end{equation}
So 
\begin{equation}
\lim \limits_{\substack{{ k\to \infty}\\ {x^k \to x^*}  \\ { \alpha \to 0^+} }} \text{sup} f\left(x^k+\alpha( \frac{s^k}{\|s^k\|})\right) > \lim \limits_{\substack{{ k\to \infty}\\ {x^k \to x^*}  \\ { \alpha \to 0^+} }} \text{sup} f(x^k+\alpha \tilde{d}),
\end{equation}
which means as $k \to \infty$, $\Delta_k \to 0$, $\Omega_k \to \Omega^* \approx \{x^*\}$, then $\frac{s^k}{\|s^k\|} \to \tilde{d}$. This contradicts these assumptions that $s^k \to s^*$ and $\tilde{d} \neq \frac{s^*}{\|s^*\|}$. Thus the desired result is established.
\end{proof}

\subsection{Neural-trust-region algorithm using black-box model}

We now relax the model degree assumption, which means the model $m_k^{NN}$ is not required to be quadratic. However, if we still train and test  the model on the interpolation set, there might be a few drawbacks. First, the trained model $m_k^{NN}$ is not robust and is very sensitive to a small change to the interpolation data set $Y$. In other words, at every iteration, the model must be constructed from scratch, which is computationally expensive. To moderate or eliminate this issue, we split the interpolation set $Y$ into two subsets, training set $Y_{\text{Tr}}$ and test set $Y_{\text{Te}}$,
\begin{equation}
\begin{aligned}
Y_{\text{Tr}}=\left\{ y^{r_1}, y^{r_2}, \dots, y^{r_q} \right\}\\
Y_{\text{Te}}=\left\{ y^{e_1}, y^{e_2}, \dots, y^{e_w} \right\},
\end{aligned}
\end{equation}
where $ y^{r_i}$, $  y^{e_j} \in Y$ and $r_q+e_w=p$. Then the model gets trained on the training set $Y_{\text{Tr}}$ but the loss function gets minimized on $Y_{\text{Te}}$. Note that when we used a neural network to construct a quadratic model in subsection \ref{subsec:unsupervised}, we considered the entire interpolation set $Y$ as the training set $Y_{\text{Tr}}=Y$ and minimized the loss function on $Y_{\text{Te}}=Y$. Hence, the error of the neural-network tends to zero as we increase the number of iterations. While it is not true for the case that we test the model on the test set, where the minimum of loss function might be nonzero even if the number of epochs is arbitrarily increasing. Splitting the interpolation data set into training set and test set bring the variance of the model under control to some extent. It means at every iterate we can update the previous model rather than constructing it again from scratch. 

One might look for the benefit of using a higher-order and black-box model over the quadratic model. A key benefit is the ability to inexpensively model general smooth functions along with derivative information. Moreover, nonquadratic models can more closely approximate the objective function within a trust-region. that might reduce the overall number of evaluations for majority of problems. In order to have a balanced loss function, we might define a new loss that is the weighted avarage of two different loss functions for the points on the boundary of the trust-region and the points lying strictly inside the trust-region.

In order to find a valid model and solve the corresponding subproblem, we can build a \textit{parent-child} net in which the weights of the child depend on the parent weights. A parent neural-network is built to find an initial model of the objective function and the corresponding subproblem while a child neural-network is built to solve the subproblem and tune the weights of the parent neural-network in order to obtain a solution $s^k$ with sufficient reduction on the objective function. Since the neural-trust-region method is based on black-box model, we might look for a balanced model for the objective function on the boundary and within the trust-region. It means, the trained model has a better agreement throughout the trust-region. To accomplish this goal, we train a model by minimizing   
the mean squared error loss 
\begin{equation}\label{loss}
                MSE_{lb}=MSE_{w}+MSE_{b},
\end{equation}
with
\begin{equation}               
 MSE_{w}=\frac{1}{n_w}\sum_{i=1}^{n_w}\left[m_k(s_i)-f(s_i) \right]^2,
 \end{equation}
and 
\begin{equation}
 MSE_{b}=\frac{1}{n_b}\sum_{j=1}^{n_b} \left[m_k(t_j)-f(t_j) \right]^2, 
  \end{equation}
 where 
\begin{equation}\label{ap124}
   \begin{cases}
                                   m_k(s_i)=f(s_i)  & \text{for }  s_i \in S_k \text{ and } i=1, 2, \dots, n_w, \\
                                 m_k(t_j)=f(t_j) & \text{for } t_j \in T_k\text{ and } j=1, 2, \dots, n_b, 
  \end{cases}
\end{equation}
with 

\begin{equation}
S=\left\{y^i \in Y~|~\|y^i -x^k\|< \Delta_k \right\},
\end{equation}
and
\begin{equation}
T=\left \{y^j \in Y~|~\|y^j -x^k\|= \Delta_k \right\},
\end{equation}
respectively. Note that at iterate $k$, if $T$ is empty, then we can sample boundary poits or train the model on only $S$. The the child neural-network becomes active to solve
\begin{equation}\label{may32}
\min \limits_{\|s\| \leq \Delta_k}m_k^{NN},
\end{equation}
where $m_k^{NN}$ is not necessarily quadratic. The child neural-network seeks $x^+=x^k+s^k$ at which not only does the model decrease but also there is a good agreement between the model and the objective function, which is 
\begin{equation}\label{fff1}
\frac{f(x^k) - f({ x^k}^{+})}{ m^{NN}_k ( x^k )-m^{NN}_k( {x^k}^{+} )}>\eta_1
\end{equation}
for a local minimizer, where $\eta_1>0$. We can find $s^k$ satisfying (\ref{fff1}) by training a net to solve 
\begin{equation}
s^k=\arg \min_{s} \left\|\frac{f(x^k) - f( x^k + s)}{ m^{NN}_k ( x^k )-m^{NN}_k( x^k+s )} -\eta_3 \right\|,
\end{equation}
 where $\eta_3>1$ is predetermined. If $\rho_k>\eta_3$, then $x^k+s^k $ is called a ``too successful iteration,'' a concept introduced by J. Walmag et al. \cite{walmag2005note}.
 
A child neural-network might simply find a local minimzer of $m_k^{NN}$ in the current trust-region at which the agreement between the model and the objective function is good. We can find a local minimizer of the subproblem by applying KKT conditions and the augmented Lagrangian method. The Lagrangian function of the subproblem is
\begin{equation}
L(s, \lambda)=m_k^{NN}+\frac{\lambda}{2} \left(\|s\|^2-\Delta_k^2 \right),
\end{equation}
where $\lambda\geq 0$, and

\begin{equation}
\begin{aligned}
\nabla_s L(s,\lambda)=\nabla_s m_k^{NN}+\lambda s,\\
\nabla_{\lambda} L(s,\lambda)=\|s\|^2- \Delta_k^2.
\end{aligned}
\end{equation}

So we look for $\lambda^*$ and $s^*$ such that 
\begin{equation}\label{may31}
\begin{aligned}
\nabla_s L(s^*,\lambda^*)=0,\hspace{1cm}\\
\|s^*\|^2\leq\Delta_k^2,\hspace{1.3cm}\\
\lambda^*(\|s^*\|^2-\Delta_k^2)=0, \text{ and}\hspace{.5cm}\\
 \nabla_{ss} L(s^*,\lambda^*) \text{ is positive definite. }
\end{aligned}
\end{equation}

Suppose $s^*$ is the solution of (\ref{may32}). If $s^* < \Delta_k$, then $s^*$ is the global minimizer of the unconstrained problem
\begin{equation}
\min \limits_{s \in \mathbb{R}^n}m_k^{NN},
\end{equation}
where $m_k^{NN}$ is not necessarily quadratic. So 
\begin{equation}\label{may33}
\begin{aligned}
\nabla_sm_k^{NN}(s^*)=0, \text{ and}\hspace{1cm} \\
\nabla_{ss}m_k^{NN}(s^*) \text{ is positive semidefinite}.
\end{aligned}
\end{equation}
 If $s^* = \Delta_k$, then $s^*$ is the minimizer of the constrained problem
\begin{equation}
\min \limits_{s = \Delta_k}m_k^{NN}.
\end{equation}
So by the second necessary optimality conditions, we have
\begin{equation}\label{may34}
\begin{aligned}
\nabla_s  L(s^*,\lambda^*)=\nabla_s m_k^{NN}(s^*)+\lambda^* s^*=0, \hspace{1cm}\\
\nabla_{ss}  L(s^*,\lambda^*)=\nabla_{ss} m_k^{NN}(s^*)+\lambda^* I \text{ is positive semidefinite.}
\end{aligned}
\end{equation}
On the other hand, if $0 \neq \|s^*\|<\Delta_k$, which means $\|s^*\|-\Delta_k \neq0$, then $\lambda^*$ must be equal to zero so that the first conditions in (\ref{may33}) and (\ref{may34}) hold. $\lambda^*$ must be nonnegative so that the optimality conditions remain valid. Moreover, if $\lambda^*>0$, then the first scenario, $\|s^*\|<\Delta_k$, cannot happen and that means $\|s^*\|-\Delta_k=0$. So, $s^*$ and $\lambda^*$ must satisfy
\begin{equation}
\lambda^*(\|s^*\|-\Delta_k)=0.
\end{equation}
Therefore, we can summarize the optimality conditions for the model $m_k^{NN}$, which is not necessarily quadratic as follows:
\begin{equation}\label{may38}
\begin{aligned}
\nabla_s m_k^{NN}(s^*)+\lambda^* s^*=0, \hspace{1.5cm}\\
\nabla_{ss} m_k^{NN}(s^*)+\lambda^* I \text{ is positive semidefinite, and}\\
\lambda^*(\|s^*\|-\Delta_k)=0.\hspace{1.8cm}
\end{aligned}
\end{equation}

To fulfill this task, the child neural-network seeks a pair $(s,w^*)$, where $w^* \geq0$, satisfying the optimality conditions, 
\begin{equation}\label{may36}
 L_s=\zeta_1 \left[\nabla_s m_k^{NN}(s)+w^* s\right]^2+\zeta_2 \left[w^*(\Delta_k-\|s\|)\right]^2+\zeta_3 \left[\hat{\lambda}_1-\tilde{c}\right]^2,
 \end{equation}
where $ \zeta_1\geq0,$ $ \zeta_2\geq0$ and $ \zeta_3\geq0$ are the weights for each term, $\hat{\lambda}_1$ is the smallest eigenvalue of $\nabla_{ss}m_k^{NN}+w^*I$, $\tilde{c}\geq 0$ is a predetermined constant, 
\begin{equation}\label{m91}
               s=\sum_{i=1}^{n} w^{s}_i e_i.
 \end{equation}
 and
 \begin{equation}\label{f22}
L_{a}=[\frac{f(x^k) - f( x^k + s)}{ m_k ( x^k )-m_k( x^k+s )} -\eta_3]^2.
 \end{equation} 
 
Note that the parent net weights will be actively adjusted while the child neural-network is searching for the minimizer of  (\ref{may32}). So the ultimate goal might be to minimize the loss 
\begin{equation}\label{Loss}
\text{Overall-Loss}=MSE_{lb} + L,
\end{equation}
where $MSE_{lb}$ is defined in (\ref{loss}) and $L$ is defined as
\begin{equation}\label{ls}
                  L=\tilde{\zeta_1} L_{s}+\tilde{\zeta_2}L_{a},
 \end{equation}
where $\tilde{\zeta_1}$ and $ \tilde{\zeta_2}\in [0,1]$ are preassigned constants.

However, training a model that satisfies (\ref{may38}) is too expensive due to $\nabla_s m_k^{NN}$ and $\nabla_{ss} m_k^{NN}$ calculations. For $x^k=(x_1^k, x_2^k, \dots, x_n^k)$ and $s^k=(s_1^k, s_2^k, \dots, s_n^k)$, we have
\begin{equation}\label{117}
m_k^{NN}(x^k+s^k) = m_k^{NN}(x^k)+(s^k)^T \nabla_x m_k^{NN} (x^k) +\frac{1}{2}(s^k)^T\nabla_{xx}m_k^{NN}(x^k)(s^k) +R_{{x^k},2}(s^k),
\end{equation}
where $R_{{x^k},2}(s)$ is the Lagrange remainder term,
\[R_{{x^k},2}(s^k)=\sum_{r_1+r_2+\dots+r_n=3} \left(\frac{\left(s_1^k\right)^{r_1}\left(s_2^k\right)^{r_2}\dots(s_n^k)^{r_n}}{r_1! r_2! \dots r_n!} \right) \left(\frac{\partial ^3 f(x^k+ts^k)}{\partial^{r_1} x_1\partial^{r_2} x_2\dots\partial^{r_n} x_n}\right),
\]
for some $t \in (0,1)$. Moreover,
\begin{equation}
\left| R_{{x^k},2}(s^k) \right| \leq \frac{\kappa_{ug}}{6}\left(\|s^k\|_1\right)^3 \leq \frac{\kappa_{ug}}{6} \left(\sqrt{n}\|s^k\|_2 \right)^3\leq \frac{\kappa_{ug}}{6}\left(\sqrt{n}\Delta_k \right)^3,
\end{equation}
where
\begin{equation}
\left|\frac{\partial^3 m_k^{NN}}{\partial x_p^k\partial x_q^k\partial x_r^k} \right| \leq \kappa_{ug} \hspace{.3cm} \text{for all }\hspace{.1cm}  p, q, r =1, 2, \dots, n.
\end{equation} 
So from (\ref{117}), we have
\begin{equation}
\nabla_{s}m_k^{NN}(s^k)= \nabla_x m_k^{NN} (x^k)+(s^k)^T\nabla_{xx}m_k^{NN}(x^k)+\nabla_{s} R_{{x^k},2}(s),
\end{equation}
and
\begin{equation}
\nabla_{ss}m_k^{NN}(s^k)=\nabla_{xx}m_k^{NN}(x^k)+\nabla_{ss} R_{{x^k},2}(s).
\end{equation}

On the other hand, a neural-network is a powerful tool that means we should not misuse its capabilities by putting unnecessary restriction on it. The aim of (\ref{may38}) is to find the global minimizer of a proper model of the objective function in the current trust-region. We can simply achieve this goal by constructing a parent-child net and defining an appropriate loss function. 

In traditional trust-region methods, the trust-region boundary is a safegaurd to prevent unbounded subproblem arisen from a nonconvex model. We might relax trust-region boundary in a controlled way, which is different from Lagrangian and penalty methods. Suppose at iterate $x^k$, a neural-network is trained to satisfyan appropriate loss function, and finds a point $x^{\text{out}}=x^k+s^{\text{out}}$ outside of the trust region at which the model $m_k^{NN}$ meets a local minimum and the agreement between the model and the objective function remains very good, then we might accept it as the new iterate $x^+=x^{\text{out}}$. This strategy does not violate any convergence conditions and we might update $\Delta_{k+1}=\gamma\|x^k-x^+\|$, where $\gamma>0$. It means, we train likely a nonconvex model, which is continuously differentiable, based on the data points in the current trust-region and we then look for a local minimum of the model at which the objective function value decreases significantly. This approach might reduce the number of model training, which means it might be more efficient for some problems. Thus, a child neural-network might simply replace (\ref{ls}) by  
\begin{equation}
                  L'=L'_{s}+L'_{a},
 \end{equation}
where
 \begin{equation}\label{f21}
 L'_{s}=\left[\frac{1}{n}\sum_{i=1}^{n}\left(\frac{\partial m_k^{NN}(x^k+s^k)}{\partial x_i}\right)^2\right]+ \frac{1}{n}\sum_{i=1}^{n}(H_i^m-c_i)^2,
 \end{equation}
and 
 \begin{equation}\label{f22}
L'_{a}=\left[\frac{f(x^k) - f( x^k + s)}{ m_k ( x^k )-m_k( x^k+s )} -\eta_3\right]^2, 
 \end{equation}      
where $c_i>0$ are preassigned constants, $g_m$ is the gradient of the model and $H_i^m$ is the $i^{\text{th}}$ leading principal minor of $\nabla_{xx}m_k^{NN}$ obtained through backpropagation,
\begin{equation}
H_i^m = det\left(\left[ \begin{array}{cccc} w^H_{11} & w^H_{12} &\hdots & w^H_{1i}\\ w^H_{21} & \ddots & \ddots & \\ & \ddots & \ddots & w^H_{(i-1)(i-1)}\\ w^H_{i1} & \hdots &w^H_{i(i-1)}&w^ H_{ii}
 \end{array} \right]\right).
\end{equation}
Note that in (\ref{f21}) the first term is to assure the gradient of the model is close to $0$ and the second term is to assure the model is concave up at $x^k+s^k$. The second term in (\ref{f21}), which checks the concavity of the function, might be replaced with $\lambda''_1-c$, where $c>0$ and $\lambda''_1$ is the smallest eigenvalue of $\nabla_{xx}m^{NN}_k$, but for training a model by neural-network, the leading principal minors of the Hessian might be beneficial.

For the sake of convergence analysis, we now define an appropriate loss function $L_{BNTR}(W)$ for a neural-trust-region based on a black-box model. At iterate $x^k$, the loss function $L_{BNTR}(W)$ aims to find a point $x^+ \in \Omega_k$ that satisfies

\begin{equation}\label{mayy32}
\begin{aligned}
\|x^k-x^+\|\leq \Delta_k, \hspace{1.1cm}\\
m_k^{NN}(x^+) \leq m_k^{NN}(x^k) - \beta' \|s^k\|^2 ,\text{ and}\\
\rho_k=\frac{f(x^k)-f(x^+)}{m_k^{NN}(x^k)-m_k^{NN}(x^+)}\geq \eta_2,
\end{aligned}
\end{equation}
where $\beta' >0$ is a preassigned constant. We can restate (\ref{mayy32}) as 
\begin{center}
$\|s^k\|\leq \Delta_k, $\\
$m_k^{NN}(x^+) = m_k^{NN}(x^k) - \beta' \|s^k\|^2 +\beta'',\text{ and}$\\
$\rho_k=\eta_2+\eta'=\eta'',$
\end{center}
where $\beta'' \geq 0$ and $\eta'\geq 0$. The loss function is defined as follows:
\begin{equation}\label{may42}
L_{BNTR}=\gamma_1L_{\Delta_k}+\gamma_2\left\|m_k^{NN}(x^k+s^k) - m_k^{NN}\left(x^k\right) + \beta' \|s^k\|^2 -\beta'' \right\|+\gamma_3 \left \|\rho_k-\eta'' \right \|,
\end{equation}
where 
\begin{equation}\label{may772}
 L_{\Delta_k}=
\begin{cases}
\left\|s^k\right\|\left(e^{\left(\left\|s^k\right\|-\Delta_k\right)}-1 \right)-\frac{\left\|s^k\right\|^2}{2}+\frac{\Delta_k^2}{2}& \text{ if }\left\|s^k\right\|>\Delta_k,\\
0& \text{ if }\left\|s^k\right\|\leq\Delta_k.
\end{cases}
\end{equation}
with $\gamma_1, \gamma_2, \gamma_3 \geq0$ are preassigned weights. The first term on right-hand side of (\ref{may42}), $ L_{\Delta_k}$ controls the step legnth $\|s^k\|$ and is twice differentiable. The second term on right-hand side of (\ref{may42}), which significantly depends on values of $\beta'$ and $\beta''$, aims to find a step $s^k$ at which 
\[m_k^{NN}(x^k)-m_k^{NN}(x^k+s^k)\geq \beta \|s^k\|^2,\] 
where $\beta>0$, which means the step $s^k$ is accepted by the child neural-network if the obtained reduction along $s^k$ is at least as small as Cauchy reduction. 

The third term on right-hand side of (\ref{may42}), which significantly depends on values of $\eta''$ seeks a step $s^k$ at which the agreement between the model and the objective function is almost equal to $\eta''$. For instance, if $\eta''=1$, the child neural network looks for a very successful step $s^k$. We might select $\gamma_3$ to be zero, and let the algorithm adjust the trust-region radius if it is needed. 

However, we look for a step $s^k$ that reduces the model value more than the Cauchy step reduction. So the value of $\beta^{'}$, in the second term on right-hand side of (\ref{may42}), plays a key role to find an efficient step $s^k$. Since we use a black-box model in Algorithm \ref{alg:ALG2}, a local minimizer of the model is not a global minimizer in the trust-region. But if the agreement between the model and the objective function at $x^k+s^k$ is very good ($\rho_k \approx 1$) and $x^k+s^k$ is a local minimizer of the model (satisfies second order sufficient optimality conditions), then $s^k$ will lead us to the local minimizer of the objective function. We can also replace the third term on right-hand side of (\ref{may42}) with
\begin{equation}\label{may2525-1}
L_{\text{agreement}}=\cosh (\rho_k-\eta^{''}),
\end{equation}
where $\cosh(.)$ is Hyperboic cosine. So we can modify our search if we use the loss function,
\begin{equation}\label{may26-1}
\begin{aligned}
L^{*}_{BNTR}=\gamma_1L_{\Delta_k}+\gamma_2L_{\text{Cauchy} }+ \gamma_3 L_{\text{local}}+\gamma_4 L_{\text{agreement}},
\end{aligned}
\end{equation}
where $\gamma_1, \gamma_2, \gamma_3, \gamma_4 \geq 0$, $L_{\Delta_k}$ is defined in (\ref{may772}),
\begin{equation}
\begin{aligned}
L_{\text{Cauchy} }=&\left\|m_k^{N}(x^k+s^k) - m_k^{N}\left(x^k\right) + \beta' \|s^k\|^2 -\beta'' \right\|,\\
L_{\text{local}}=&\gamma_2' \left\|\lambda''_1-c\right\| +\gamma_2^{''} \frac{1}{n}\sum_{i=1}^{n}\left(\frac{\partial m_k^{N}(x^k+s^k)}{\partial x_i}\right)^2,
\end{aligned}
\end{equation}
with $\gamma_2', \gamma_2^{''} \geq 0$, $c>0$, $\lambda''_1$ is the smallest eigenvalue of $\nabla_{xx}m^{N}_k$ and $L_{\text{agreement}}$ is defined in (\ref{may2525-1}).

Values of the hyper-parameters  $\gamma_1, \gamma_2, \gamma_3, \gamma_4, \gamma_2' $ and $\gamma_2{''}$ play a key role in the efficiency of the algorithm. For instance, if we are so ambitious and set $\gamma_1=\gamma_3=\gamma_4=M$ and $\gamma_3=0$, where $M$ is a very big number, then the neural-network trust-region algorithm looks for a local minimum of the model within the current trust-region at which the agreement between the model and the objective function is very good. There might not be such a point in the current trust-region, which means the algorithm will not terminate. In this case, we can enlarge the trust-region and sample more points and train a new model and repeat the procedure. But in practice, we do not choose a \textit{Big-M} as the value of hyper-parameters. 
\begin{algorithm}[H]
            \caption{Neural-trust-region algorithm based on a smooth black-box model} 
            \label{alg:ALG2}
            \begin{flushleft}
                \textbf{Step 0: Initialization.} 
                 An initial point $x_0$ and an initial trust-region radius $\Delta_0>0$ are given. The constants  $0\leq \eta_1 \leq \eta_2 \leq1\leq \eta_3$, $0<\gamma_1  <\gamma_2 \leq1 \leq \gamma_3$, $\epsilon>0$, positive integers $n_w$ and $n_b$, which are the initial number of the collocation points for $f(x)$ within the trust-region and training points on the trust-region boundary are also given. Compute $f(x_0)$ and set $k=0$.\\
             \end{flushleft}    
             
             \begin{flushleft}
                \textbf{Step 1: Model definition and Step calculation.}  Sample or update $S=\{s_1,s_2, \dots, s_{n_w}\}$ within trust-region and $T=\{t_1,t_2, \dots, t_{n_b}\}$ on the boundary of the trust-region centered $x^k$ with radius $\Delta_k$. Applying \textbf{Algorithm \ref{alg:ALGmo}}, calculate a trial step $s^k$.
                 
             \end{flushleft}     
             
             \begin{flushleft}
                \textbf{Step 2: Acceptance of trial point.} Compute $f(x^k+s^k)$ and 
                 $$\rho_k = \frac{f(x^k) - f( x^k + s^k )}{ m_k ( x^k )-m_k( x^k+s^k )}.$$
   \hspace{1.2cm}           If $\rho_k\geq \eta_1$, then define $x^{k+1}=x^k+s^k$; otherwise define $x^{k+1}=x_{k}$.

             \end{flushleft}    
             
             \begin{flushleft}
                \textbf{Step 3: Trust-region radius and training set update.} 
                
                 \hspace{1.2cm}   If $\rho_k\geq \eta_2$, then $\Delta_{k+1}=\gamma_3 \Delta_k$, replace $y^- \in Y=S \cup T$ with $x^+=x^k+s^k$, where 
                 
                    \hspace{1.2cm}   $y^-$ is defined in (\ref{y-}).

 \hspace{1.2cm}           If $\rho_k\geq \eta_1$, then $\Delta_{k+1}=\gamma_2 \Delta_k$ and replace $y^- \in Y=S \cup T$ with $x^+=x^k+s^k$, where 
 
 \hspace{1.2cm}  $y^-$ is defined in (\ref{y-}).
 
 \hspace{1.2cm}          If $\rho_k< \eta_1$ \textit{and} the interpolation set $Y$ is inadequate, then for every $y^i \in Y$ find $y_{pr}^i$ by (\ref{ap122}). Then select the 
 
 \hspace{1.2cm}  exiting point $y^-$ by (\ref{ap121}). Finally, replace $y^-$ with a point $y^+$ in the trust-region to improve $D(Y)$ in (\ref{ap123}).
 
  \hspace{1.2cm}           If $\rho_k< \eta_1$ \textit{but} the interpolation set $Y$ is adequate, then $\Delta_{k+1}=\gamma_1 \Delta_k$. More interpolation points  
 
   \hspace{1.2cm}    can be sampled.

\hspace{.2cm}Increment $k$ by $1$ and go to Step 1.
                \end{flushleft}        
    \end{algorithm}
    
    \begin{algorithm}[H]
            \caption{Model definition and Step calculation } 
            \label{alg:ALGmo}
             \begin{flushleft}
                \textbf{Step 0: Initialization.} The current point $x^k$ and the current trust-region radius $\Delta_k>0$ are given. The constant $\eta_3\geq1$, $S=\{s_1,s_2, \dots, s_{n_w}\}$ and $T=\{t_1,t_2, \dots, t_{n_b}\}$ are also given. \\
             \end{flushleft}  
            \begin{flushleft}
                \textbf{Step 1: Spliting the data.} Split $S=\{s_1,s_2, \dots, s_{n_w}\}$ into the training set $S_{\text{Tr}}=\{s_1,s_2, \dots, s_{n_w'}\}$ and the test set $S_{\text{Te}}=\{s_1,s_2, \dots, s_{n_w''}\}$. Split $T=\{t_1,t_2, \dots, t_{n_b}\}$ into the training set $T_{\text{Tr}}=\{t_1,t_2, \dots, t_{n_b'}\}$ and the test set $T_{\text{Te}}=\{t_1,t_2, \dots, t_{n_b''}\}$.\\
             \end{flushleft}

             \begin{flushleft}
                \textbf{Step 2: Training the model.}  Construct a model $m_k(x)$ for $f(x)$ within the current trust-region with radius $\Delta_k$ on the training interpolation set and minimize the loss function on the test interpolation set. We use MSE as the loss function and minimize 
                \begin{equation}
                MSE=MSE_{w}+MSE_{b},
                \end{equation}
where                
               $$MSE_{w}=\frac{1}{n_w''}\sum_{i=1}^{n_w''}[m_k(x_i)-f(x_i)]^2,$$
and 
                $$MSE_{b}=\frac{1}{n_b''}\sum_{i=1}^{n_b''}[m_k(x_i)-f(x_i)]^2, $$               
              \end{flushleft}   
on the test interpolation sets. For any point $x=(x_1, x_2, \dots, x_n)$, all partial derivatives with respect to $x_1, x_2, \dots, x_n$ of the model obtained in step 1  are available through backpropagation. 
                 \begin{flushleft}
               
                \textbf{Step 3: Step Calculaton.} Now, we look for a trial step $s^k$ that minimizes the model obtained in step 1. Suppose 
               \begin{equation}
               s^k=\sum_{i=1}^{n} w_i^g e_i,
               \end{equation}
               where $e_i$'s are the standard unit vectors in $n$ dimensions. All we need is to determine values of $w_i^g$, for $i=1, 2, \dots, n$, such that an appropriate loss function for a parent-child net such as one that is defined in (\ref{may26-1}) meets its minimum or an approximate of it. 
             \end{flushleft}   
                       
         \end{algorithm}   

We can apportion $Y$ into training and test sets, with an $80 \%-20 \% $ split, which means $20 \% $ of the interpolation points that should be held over for testing. If we relax the quadratic model assumption, we need to have enough interpolation data points to train a robust model. If the cardinality of interpolation set is not large enough, we might use machine-learning techniques such as $k$-fold cross-validation. In \textit{Python}, we can import the \textit{train-test-split} and the \textit{cross-val-score} from \textit{sklearn.model-selection} library to split $Y$ into training and test sets, and use cross-validation technique, respectively. As an alternative, we might sample more interpolation points to find the model minimizer that reduces the objective function sufficiently if it is needed.

\subsection{Convergence analysis of Algorithm \ref{alg:ALG2}}
In Algorithm \ref{alg:ALG2} the trained neural-network model $m_{k}^{NN}$ is not required to be quadratic, and the rest of its structure is almost identical to Algorithm \ref{alg:ALG1}. So we expect the convergence analysis results of Algorithm \ref{alg:ALG1} remain valid for Algorithm \ref{alg:ALG2}. At each iteration, we still use differentiable loss and activation functions to train a neural-network model, which means Lemma \ref{NTR:lemma1} and Lemma \ref{NTR:lemma2} are still valid for Algorithm \ref{alg:ALG2}.

\begin{lemma} \label{may48}
In Algorithm \ref{alg:ALG2}, where a black-box model and the loss function given by (\ref{may42}) are employed to train a neural-network model $m_k^{NN}$, for all $k$ and for some $\bar{c} >0$ the subproblem solution satisfies 
\begin{equation}\label{emrooz2appp}
\left[m^{NN}_k(x^k)-m^{NN}_k(x^+)\right]\geq \bar{c}\left[m_k(x^k)-m_k(x^k +s^C\right],
\end{equation}
where $s^C=-t_k^C \nabla m_k(x^k)$ is the Cauchy step.
\end{lemma}
\begin{proof}
The trained model satisfies (\ref{mayy32}), so we have
\begin{equation}\label{may41}
m_k^{NN}(x^k)-m_k^{NN}(x^+)\geq \beta' \|s^k\|^2,
\end{equation}
and since $\|s^k\| \leq \Delta_k$, there exists ${c} \in (0,1]$ such that $\|s^k\| =\sqrt{{c}} \Delta_k$. So from (\ref{may41}), we have
\begin{equation}
m_k^{NN}(x^k)-m_k^{NN}(x^+)\geq \beta' {c} \Delta_k^2,
\end{equation}
which means the model reduction is of order $\Delta_k^2$. On the other hand, without loss of generality, if we assume $\Delta_k \approx \|\nabla m_k(x^k)\|$, then the the model decrease attained by the Cauchy step in (\ref{emroozap1}) is of order $\Delta_k^2$. Therefore, there exists $\hat{c} >0$ that satisfies (\ref{emrooz2appp}).
\end{proof}

\begin{lemma}\label{radiustozero1}
Any sequence $\{ \Delta_k\}$ produced by Algorithm \ref{alg:ALG2} satisfies
\begin{equation}
\lim\limits_{k \to \infty} \Delta_k=0.
\end{equation}
\end{lemma}
\begin{proof}
The proof is identical with the one for Lemma \ref{radiustozero}.
\end{proof}

Lemmas \ref{may48} and \ref{radiustozero1} guarantee the validity of Theorems \ref{NTR-Theorem1} and \ref{NTR-Theorem2} for Algorithm \ref{alg:ALG2}.

\section{Conclusion}
In this paper (part 1), we introduced a new derivate-free trust-region method in which an artificial neural-network is used to approximate a model of the objective function within the trust-region and solve the corresponding subproblem.

\bibliographystyle{unsrt}  

\end{document}